\documentclass[letterpaper]{amsart}
\usepackage{amsfonts,amssymb,amscd,amsmath,enumerate,url,verbatim}
\usepackage[all]{xy}

\theoremstyle{plain}
\newtheorem{theorem}{Theorem}[section]

\newtheorem{prop}[theorem]{Proposition}
\newtheorem{lemma}[theorem]{Lemma}
\newtheorem{corollary}[theorem]{Corollary}
\newtheorem{notation}[theorem]{Notation}
\newtheorem{convention}[theorem]{Convention}
\newtheorem{definition}[theorem]{Definition}
\newtheorem{question}[theorem]{Question}
\newtheorem{observation}[theorem]{Observation}
\newtheorem{example}[theorem]{Example}

\begin{document}
\title[On  the degrees of relations on $x_1^{d_1}, \ldots, x_n^{d_n}, (x_1+ \ldots + x_n)^{d_{n+1}}$ in positive characteristic]
{On the degrees of relations on $x_1^{d_1}, \ldots, x_n^{d_n},  (x_1+ \ldots + x_n)^{d_{n+1}}$ in positive characteristic}
\date{\today}

\author{Adela Vraciu}
\address{Adela Vraciu\\Department of Mathematics\\University of South Carolina\\\linebreak 
Columbia\\ SC 29208\\ U.S.A.} \email{vraciu@math.sc.edu}


\thanks{Research partly supported by NSF grant  DMS-1200085}
\maketitle
\begin{abstract}
We give a formula for the smallest degree of a non-Koszul relation on $x_1^{d_1}, \ldots, x_n^{d_n}, (x_1+\ldots +x_n)^{d_{n+1}}\in k[x_1, \ldots, x_n]$ (under certain assumptions on $d_1, \ldots, d_{n+1}$) where $k$ is a field of positive characteristic $p$. As an application of our result, we give a formula for the diagonal F-threshold of a diagonal hypersurface. Another application is a characterization, depending on the characteristic $p$ of $k$,  of the values of $d_1, \ldots, d_{n+1}$ (satisfying certain assumptions) such that the ring $k[x_1, \ldots, x_{n+1}]/(x_1^{d_1}, \ldots, x_{n+1}^{d_{n+1}})$ has the weak Lefschetz property.
\end{abstract}

\section{Introduction}
The goal of this paper is to answer the following question:
\begin{question}\label{Question}
Let $k$ be a field of characteristic $p\ge 0$ and let $d_1, \ldots, d_{n+1}$ be positive integers. What is the smallest degree of a non-Koszul relation on the elements $x_1^{d_1}, \ldots, x_n^{d_n}, (x_1+ \ldots + x_n)^{d_{n+1}}$ in the polynomial ring $k[x_1, \ldots, x_n]$? Equivalently, what is the smallest degree of a non-zero element in 
$$
\frac{(x_1^{d_1}, \ldots, x_n^{d_n}):(x_1+\ldots + x_n)^{d_{n+1}}}{(x_1^{d_1}, \ldots, x_n^{d_n})}?
$$
\end{question}
We will assume throughout that $d_1, \ldots, d_{n+1}$ are such that none of the elements $x_1^{d_1}, \ldots, x_n^{d_n}, (x_1+ \ldots + x_n)^{d_{n+1}}$ is in the ideal generated by the others. If $\mathrm{char}(k)=0$, this is equivalent to 
\begin{equation}\label{condition}
d_i \le \sum_{j\ne i} (d_j-1) \  \ \forall i \in \{1, \ldots, n+1\}.
\end{equation}
 The condition for $\mathrm{char}(k)=p>0$ is more complicated.
We introduce the following notation:
\begin{notation}
Let $f:=x_1 + \ldots + x_n$.

Let ${\mathcal A}=(a_1, \ldots, a_{n+1})^t$ be a homogeneous relation on $x_1^{d_1}, \ldots, x_n^{d_n}, f^{d_{n+1}}$,( i.e. $a_1x_1^{d_1} + \ldots + a_nx_n^{d_n}+a_{n+1}f^{d_{n+1}}=0$).
We define the degree of the relation ${\mathcal A}$ to be $\mathrm{deg}({\mathcal A}):=\mathrm{deg}(a_i)+d_i$ {\rm(}note that this does not depend on $i \in \{1, \ldots, n+1\}${\rm)}.

Let ${\mathcal E}_p(d_1, \ldots, d_{n+1})$ be $\mathrm{min}(\mathrm{deg}({\mathcal A}))$, where ${\mathcal A}$ is a non-Koszul relation, i.e. $a_i \notin (x_1^{d_1}, \ldots, x_{i-1}^{d_{i-1}}, x_{i+1}^{d_{i+1}}, \ldots, f^{d_{n+1}})$ for some (equivalently for all) $i \in \{1, \ldots, n+1\}$.
\end{notation}

All the relations considered in this paper will be homogeneous relations. Since the degrees of the relations are not affected by a flat base change, we may assume without loss of generality that $k$ is a perfect field.

\begin{observation}
Note that the function $\mathcal{E}_p(d_1, \ldots, d_{n+1})$ is symmetric in the variables $d_1, \ldots, d_{n+1}$. This is because one can do a change of variables $x_i=(x_1 + \ldots + x_n)-x_1 - \ldots -x_{i-1} - x_{i+1} - \ldots - x_n$ which allows one to switch $d_i$ and $d_{n+1}$.
\end{observation}

The answer to Question (\ref{Question}) is given in \cite{RRR} in the case when $\mathrm{char}(k)=0$.
\begin{theorem}[Theorem 5, \cite{RRR}]\label{ch0}
Let $d_1, \ldots, d_{n+1}$ be positive integers satisfying (\ref{condition}). Then
$$\mathcal{E}_0(d_1, \ldots, d_{n+1})=\left\lceil \frac{\sum_{i=1}^{n+1} d_i - n +1 }{2} \right \rceil.$$
\end{theorem}
Smaller degrees are possible in positive characteristic. The following is a consequence of the methods used in \cite{RRR}.
\begin{theorem}\label{one_direction}
Let $d_1, \ldots, d_{n+1}$ be positive integers satisfying (\ref{condition}). Then
$$\mathcal{E}_p(d_1, \ldots, d_{n+1}) \le \mathcal{E}_0(d_1, \ldots, d_{n+1})$$
for every prime number $p>0$.
\end{theorem}
\begin{proof}
Consider the ring $\displaystyle A=k[x_1, \ldots, x_{n+1}]/(x_1^{d_1}, \ldots, x_{n+1}^{d_{n+1}})$ where $k$ is a field of characteristic p. Let $L:=x_1 + \ldots x_{n+1}$ and consider the map $\times L : A_i \rightarrow A_{i+1}$ where $A_i$ denotes the $i$th graded piece of $A$. It is easy to see that $\times L : A_i \rightarrow A_{i+1}$ is not injective if and only if $\mathcal{E}_p(d_1, \ldots, d_{n+1}) \le i+1$. On the other hand, it is shown in \cite{RRR} (Theorem 1) that $\mathrm{dim}_k(A_i) > \mathrm{dim}_k(A_{i+1})$ when $i\ge \lceil \frac{\sum_{i=1}^{n+1} d_i - n-1}{2}\rceil$ which shows that the map $\times L: A_i \rightarrow A_{i+1}$ cannot be injective.
\end{proof}
The work in \cite{RRR}  is related to the weak Lefschetz property for monomial complete intersections $\displaystyle A=\frac{k[x_1, \ldots, x_{n+1}]}{(x_1^{d_1}, \ldots, x_{n+1}^{d_{n+1}})}$. If $\mathrm{char}(k)=0$, it was shown in \cite{St} using the hard Lefschetz theorem from algebraic geometry that all monomial complete intersections have the weak Lefschetz property. This was then reproved in \cite{W} using representation theory. The first purely algebraic prove is given in \cite{RRR} and it is a direct consequence of Theorem (\ref{ch0}).

In positive characteristic, relations of smaller degrees are possible, arising from the fact that $f^q = x_1^q + \ldots + x_n^q$ where $q=p^e$ is a power of the characteristic. This leads to frequent failure of the weak Lefschtetz property for monomial complete intersections in positive characteristic.  This is one of the applications of our work, discussed in Section 5.

Another application (see Section 4) is the computation of the diagonal F-threshold of a diagonal hypersurface $\displaystyle R=\frac{k[x_1, \ldots, x_{n+1}]}{(x_1^a+ \ldots + x_{n+1}^a)}$ in positive characteristic. F-thresholds were introduced in \cite{MTW} in connection with jumping numbers of test ideals for generalized tight closure (which are positive characteristic analogues of jumping numbers for multiplier ideals in birational geometry). They have been studied further in \cite{BMS1}, \cite{BMS2}, \cite{HMTW}; explicit formulas for F-thresholds of certain rings and certain ideals were obtained in \cite{Da}, \cite{Hernandez}. If $\frak{a}\subseteq \sqrt{J}$, then $\displaystyle c^J(\frak{a})=\lim_{q=p^e \rightarrow \infty} \frac{\mathrm{max}\{N \, | \, \frak{a}^N\not\subset J^{[q]} \}}{q}$ is the F-threshold of $\frak{a}$ with respect to $J$. The diagonal $F$-threshold of a ring is obtained when $\frak{a}=J=m$. Diagonal F-thresholds of certain rings were studied in \cite{KMY}, \cite{Jinjia}, \cite{CM}.

Question \ref{Question} has also been answered in the case when $n=2$ and $\mathrm{char}(k)=p>0$ in \cite{Han_thesis}. We mention that when $n=2$, knowing $\mathcal{E}_p(d_1, d_2, d_3)$ allows one to completely describe the free resolution of $(x_1^{d_1}, x_2^{d_2}):(x_1+x_2)^{d_3}$ over the polynomial ring, since it is given by a Hilbert-Burch matrix consisting of two relations with degrees adding up to $d_1+d_2+d_3$. However, if $n\ge 3$, knowing the smallest degree of a relation does not allow one to draw conclusions about the other relations on the given elements. The work in \cite{Han_thesis} has been done in connection with computing Hilbert-Kunz multiplicities for diagonal hypersurfaces, see also \cite{Han_Monsky}.

The following is our main result:
\begin{theorem}\label{main_theorem}
Let $k$ be a field of characteristic $p>0$, $n\ge 3$, $R=k[x_1, \ldots, x_n]$ and $f=x_1+ \ldots +x_n$. 

{\rm (I)}
Let $q=p^e$ be a power of $p$, $0 \le r_i \le q-1$, $1 \le k_i \le p-1$, and let $d_i = k_i q+ r_i$.

 Assume that
\begin{equation}\label{cond} k_i \le \lfloor \frac{\sum_{j=1}^{n+1} k_j -n+1}{2} \rfloor \ \mathrm{ for\  all}\  i \in \{1, \ldots, n+1\}.
\end{equation}
Then 
\begin{equation}\label{answer}
\mathcal{E}_p(d_1, \ldots , d_{n+1})=\mathrm{min}\{q\mathcal{E}_p(k_1+\epsilon_1, \ldots, k_{n+1}+\epsilon_{n+1})+\sum_{\epsilon_i=0}r_i\}
\end{equation}
where the minimum is taken over all $\epsilon_1, \ldots, \epsilon_{n+1}\in \{0, 1\}$.

{\rm (II)} Let  $1\le \kappa_1, \ldots, \kappa_{n+1}\le p$ such that 
\begin{equation} \kappa_i \le \lfloor \frac{\sum_{j=1}^{n+1} \kappa_j -n+1}{2} \rfloor \ \mathrm{ for\  all}\  i \in \{1, \ldots, n+1\}.
\end{equation}. Then 
\begin{equation}\label{ch0answer}
\mathcal{E}_p(\kappa_1, \ldots, \kappa_{n+1}) = \mathrm{min}\{ \left \lceil \frac{\sum_{i=1}^{n+1} \kappa_i - n +1 }{2} \right\rceil, p\}
\end{equation}
\end{theorem}
\begin{observation}

{\bf 1.}  The inequality ($\le $) in equation (\ref{answer}) is always true provided that $k_i\ge 1$ for all $i$ (by Theorem (\ref{inequality})). However, the inequality ($\ge $) may fail when $k_1, \ldots, k_{n+1}$ do not satisfy the assumption (\ref{cond}).

{\bf 2.}  As noted in Theorem (\ref{one_direction}),  $\mathcal{E}_p(d_1, \ldots, d_{n+1}) \le \mathcal{E}_0(d_1, \ldots, d_{n+1})$. We note that equality holds if and only if the ring $\displaystyle A=\frac{k[x_1, \ldots, x_{n+1}]}{(x_1^{d_1}, \ldots, x_{n+1}^{d_{n+1}})}$ has the weak Lefschetz property {\rm (}see Corollary 2.2 in \cite{KV}{\rm )}.
\end{observation}

The following example illustrates how the result of Theorem~\ref{main_theorem} is applied in practice.
\begin{example}
\rm{We wish to calculate} $\mathcal{E}_5(6, 7, 11, 12)$. We have $k_1=k_2=1$, $k_3=k_4=2$, $r_1=r_3=1$, $r_2=r_4=2$. 
We have
$$
\mathcal{E}_p(k_1 +\epsilon_1, k_2+\epsilon_2, k_3+\epsilon_3, k_4+\epsilon_4) =\left \lceil \frac{\sum_{i=1}^4 k_i + \sum_{i=1}^4 \epsilon_i -2}{2}\right\rceil= 
 \left\lbrace \begin{array}{ccc} 2 & \mathrm{when} & \sum_{i=1}^4 \epsilon_i=0 \\ 3 & \mathrm{when} & \sum_{i=1}^4 \epsilon_i\in \{1, 2\} \\ 4 & \mathrm{when} & \sum_{i=1}^4 \epsilon_i \in \{3, 4\} \\ \end{array} \right. 
$$
We have
$$
p\mathcal{E}_p(k_1, k_2, k_3, k_4) + \sum_{i=1}^4r_i = 5\cdot 2 + 6 =16
$$
$$
\mathrm{min} \{ p\mathcal{E}_p(k_1+\epsilon_1, k_2+\epsilon_2, k_3+\epsilon_3, k_4+\epsilon_4) + \sum_{\epsilon_i=0} r_i\} = 5 \cdot 3 + 2 = 17
$$
where the minimum is over all the choices of $\epsilon_i \in \{0, 1\}$ with $\sum_{i=1}^4 \epsilon_i \in \{1, 2\}$(achieved for $\epsilon_1=\epsilon_3=0, \epsilon_2, \epsilon_4=1$) and 
$$
\mathrm{min} \{ p\mathcal{E}_p(k_1+\epsilon_1, k_2+\epsilon_2, k_3+\epsilon_3, k_4+\epsilon_4) + \sum_{\epsilon_i=0} r_i\}= 5\cdot 4 =20
$$
where the minimum is taken over the choices of $\epsilon_i \in \{0, 1\}$ with $\sum_{i=1}^4 \epsilon_i \in \{3, 4\}$ (achieved for $\epsilon_1=\epsilon_2 = \epsilon_3 = \epsilon_4=0$).

Therefore we have  $\mathcal{E}_5(6, 7, 11, 12)=16$.
\end{example}

The next example shows that the assumption (\ref{cond}) cannot be removed.
\begin{example}
{\rm We wish to calculate} $\mathcal{E}_5(7, 7, 7, 18)$. {\rm We have} $k_1=k_2=k_3=1$, $k_4=3$, $r_1=r_2=r_3=2, r_4=3$. {\rm It is easy to see that} $\mathcal{E}_5(1+\epsilon_1, 1+\epsilon_2, 1+\epsilon_3, 3)=3$ {\rm when} $\epsilon_1+\epsilon_2+\epsilon_3=2$ {\rm (since} $f^3 \in (x_1^{1+\epsilon_1}, x_2^{1+\epsilon_2}, x_3^{1+\epsilon_3}{\rm )}${\rm )}, $\mathcal{E}_5(2, 2, 2, 3)=\mathcal{E}_5(1+\epsilon_1, 1+\epsilon_2, 1+\epsilon_3, 4)=4$ {\rm for all choices of} $\epsilon_1, \epsilon_2, \epsilon_3 \in \{0, 1\}$. {\rm Therefore the  minimum that equation (\ref{answer}) yields is} $\mathrm{min}\{15+r_1+r_4, 20\}=20$. {\rm However, a Macaulay2 calculation shows that}$\mathcal{E}_5(7, 7, 7, 18)=19$.
\end{example}

Moreover, we observe that the result of Theorem~\ref{main_theorem} can only be applied to exponents $d_1, \ldots, d_{n+1}$ that have the property that the largest power of $p$ that is less than or equal to $d_i$ is the same for all $i$.

\section{Proof of the first inequality and the connection with the case $n=2$}

\begin{convention}
For convenience of notation, in the course of the proof we will refer to $f=x_1+\ldots + x_n$ by the name $x_{n+1}$. We warn the reader that this is not a new variable. 
\end{convention}
Note that every subset with $n$ elements of the set $\{x_1,\ldots, x_{n+1}\}$ is a system of parameters in $R$. Also note that one can replace $x_i$ by $x_{n+1}-x_1 -\ldots -x_{i-1}-x_{i+1} - \ldots- x_{n}$, and thus one can use any subset of $n$ elements out of $\{x_1, \ldots, x_{n+1}\}$ as variables in the polynomial ring.

\begin{theorem}\label{inequality}
Let $d_i=k_iq+r_i$ where $q=p^e, k_i \ge 1$, $0 \le r_i \le q-1$.
Then 
$$
\mathcal{E}_p(d_1, \ldots, d_{n+1}) \le q\mathcal{E}_p(k_1+\epsilon_1, \ldots, k_{n+1}+\epsilon_{n+1}) + \sum_{\epsilon_i =0} r_i
$$
for every $(\epsilon_1, \ldots, \epsilon_{n+1}) \in \{0, 1\}^{n+1}$. 
\end{theorem}
\begin{proof}
We will construct a non-Koszul relation of degree $\mathcal{E}_p(k_1+\epsilon_1, \ldots, k_{n+1}+\epsilon_{n+1})q+\sum_{\epsilon_i = 0} r_i$ for each choice of $\epsilon_1, \ldots, \epsilon_{n+1}\in \{0, 1\}$ such that if $r_j=0$ then $\epsilon_j=0$. It is enough to consider these choices because when $r_j=0$, the choice of $\epsilon_j=1$ yields a larger value of  $\mathcal{E}_p(k_1+\epsilon_1, \ldots, k_{n+1}+\epsilon_{n+1})+\sum_{\epsilon_i = 0} r_i$ than the choice of $\epsilon_j=0$.

Start with a non-Koszul relation of minimal degree $\mathcal{A}=(a_1, \ldots, a_{n+1})^t$ on $x_1^{k_1+\epsilon_1}, \ldots, x_n^{k_n + \epsilon_n}, f^{k_{n+1} +\epsilon_{n+1}}$, raise it to the $q$th power, and multiply by $x_i^{r_i}$ for each $i$ for which $\epsilon_i=0$. We obtain a relation on $x_1^{d_1}, \ldots, x_n^{d_n}, f^{d_{n+1}}$, in which the coefficient of $x_i^{d_i}$ is $a_i^q(\prod_{\epsilon_j=0, j\ne i}x_j^{r_j})$ if $\epsilon_i=0$, and it is $a_i^qx_i^{q-r_i}(\prod_{\epsilon_j=0, j\ne i}x_j^{r_j})$ if $\epsilon_i=1$. We need to show that this relation is not Koszul. If there exists an $i$ with $\epsilon_i=0$, fix such an $i$. It is enough to prove that the coefficient of $x_i^{d_i}$ is not in $(x_1^{d_1}, \ldots, x_{i-1}^{d_{i-1}}, x_{i+1}^{d_{i+1}}, \ldots, x_{n+1}^{d_{n+1}})$. 
Otherwise, we would have 
$$a_i^q(\prod_{\epsilon_j=0, j\ne i} x_j^{r_j}) \in (x_1^{k_1q+r_1}, \ldots, x_{i-1}^{k_{i-1}q+r_{i-1}}, x_{i+1}^{k_{i+1}q+r_{i+1}}, \ldots, x_{n+1}^{k_{n+1}q+r_{n+1}}).
$$
Since $x_1, \ldots, x_{i-1}, x_{i+1}, \ldots, x_{n+1}$ form a regular sequence, this implies $$a_i^q \in (x_1^{k_1q+s_1}, \ldots, x_{i-1}^{k_{i-1}q+s_{i-1}}, x_{i+1}^{k_{i+1}q+s_{i+1}}, \ldots, x_{n+1}^{k_{n+1}q+s_{n+1}})$$ where $s_j$ is equal to $r_j$ if $\epsilon_j=1$ and it is zero otherwise. Since $a_i^q$ is a $q$th power, every monomial in $a_i^q$ (where we use $x_1, \ldots, x_{i=1}, x_{i+1}, \ldots, x_{n+1}$ as variables in $k[x_1, \ldots, x_n]$) which is divisible by $x_j^{k_jp+r_j}$ will also be divisible by $x_j^{(k_j+1)q}$ if $r_j>0$. Therefore it follows that
$$a_i^q\in (x_1^{(k_1+\epsilon_1)q}, \ldots, x_{i-1}^{(k_{i-1}+\epsilon_{i-1})q}, x_{i+1}^{(k_{i+1}+\epsilon_{i+1})q}, \ldots, x_{n+1}^{(k_{n+1}+\epsilon_{n+1})q}).$$ But this implies that $$a_i \in (x_1^{k_1+\epsilon_1}, \ldots, x_{i-1}^{k_{i-1}+\epsilon_{i-1}}, x_{i+1}^{k_{i+1}+\epsilon_{i+1}}, \ldots, x_{n+1}^{k_{n+1}+\epsilon_{n+1}}),$$ which contradicts the assumption that the relation we started with was not Koszul.

Now consider the case when $\epsilon_i=1$ for all $i$. Note that this case need only be considered when $r_i>0$ for every $i$. We want to show that $a_{n+1}^qx_{n+1}^{q-r_{n+1}}\notin (x_1^{k_1q+r_1}, \ldots, x_n^{k_nq+r_n})$. Write $a_{n+1}^q=\sum_u{\alpha_u M_u}$ and $x_{n+1}^{q-r_{n+1}}=\sum_v{\beta_v N_v}$ as sums of monomials (in the variables $x_1, \ldots, x_n$) $M_u, N_v$, with coefficients $\alpha_u, \beta_v \in k\, \backslash \, \{0\}$. Note that every exponent in every $M_u$ is divisible by $q$, while every exponent in every $N_v$ is strictly less than $q$. It follows that there can be no cancellations between monomials $M_uN_v$ and $M_{u'}N_{v'}$ that occur when we multiply out $(\sum_u M_u)(\sum_v N_v)$, since no two monoials can have the same exponents for each variable unless $N_v=N_v'$ and $M_u=M_{u'}$ (we can see this by looking at the congruence classes modulo $q$ of the exponents, which are given by the exponents that occur in $N_v, N_{v'}$).

 If we assume by way of contradiction that $a_{n+1}^qx_{n+1}^{q-r_{n+1}}\in (x_1^{k_1q+r_1}, \ldots, x_n^{k_nq+r_n})$, it follows that $M_uN_v \in (x_1^{k_1q+r_1}, \ldots, x_n^{k_nq+r_n})$ for every $u, v$. In particular one can take $N_v=x_i^{q-r_{n+1}}$, and it follows that $$M_u\in (x_1^{k_1q+r_1}, \ldots, x_n^{k_nq+r_n}):x_i^{q-r_{n+1}}=
$$
$$(x_1^{k_1q+r_1}, \ldots, x_{i-1}^{k_{i-1}q+r_i}, x_i^{k_iq-q+r_i+r_{n+1}}, x_{i+1}^{k_{i+1}q+r_{i+1}}, \ldots, x_n^{k_nq+r_n}.$$
Since $M_u$ is a $q$th power and $r_i>0$, it follows that $$M_u \in (x_1^{k_1q+q}, \ldots, x_{i-1}^{k_{i-1}q+q}, x_i^{k_iq}, x_{i+1}^{k_{i+1}q}, \ldots, x_n^{k_nq+q}).$$ But this holds for every choice of $i$ and every monomial $M_u$ in the expansion of $a_{n+1}^q$. Therefore we must have 
$a_{n+1}\in \bigcap_i (x_1^{k_1+1}, \ldots, x_i^{k_i}, \ldots, x_n^{k_n +1})= (x_1^{k_1+1}, \ldots, x_n^{k_n+1}, \prod _{i=1}^n x_i^{k_i})$. 
Since the original relation $\mathcal{A}=(a_1, \ldots, a_{n+1})^t$ is not Koszul, it follows that $\mathrm{deg}(a_{n+1})\ge \sum_{i=1}^n k_i$ and thus $\mathrm{deg}(\mathcal{A})=\mathrm{deg}(a_{n+1})+k_{n+1}+1 \ge \sum_{i=1}^{n+1} k_i +1$. This contradicts the assumption that the original relation had minimal degree $\mathcal{E}_p(k_1+1, \ldots, k_{n+1}+1)$, since by Theorem (\ref{one_direction}) we have
$$\mathcal{E}_p(k_1+1, \ldots, k_{n+1}+1)\le \mathrm{max}\left\lbrace k_i+1 , \lceil \frac{\sum_{i=1}^{n+1} (k_i+1) - n +1 }{2} \rceil \right\rbrace  = 
$$
$$\mathrm{max}\left\lbrace k_i+1 , \lceil \frac{\sum_{i=1}^{n+1} k_i +2 }{2} \rceil \right\rbrace < \sum_{i=1}^{n+1} k_i +1
$$

\end{proof}

We now discuss briefly the case $n=2$. The main result in this case is Theorem 2.25 in \cite{Han_thesis}.  We do not reprove Han's result here, but rather we translate it in a more algebraic language (the original statement relies heavily on a honeycomb structure in the plane consisting of octahedra and tetrahedrons) that will allow us to compare it with our results for $n\ge 3$.

\begin{theorem}[Theorem 2.25, \cite{Han_thesis}]\label{Han}
Let $d_1, d_2, d_3$ be positive integers satisfying the triangle inequality, and let $k$ be a field of characteristic $p>0$.
Then
$$\mathcal{E}_p(d_1, d_2, d_3)=\mathrm{min}\{q \lceil \frac{\sum_{i=1}^3 (k_i + \epsilon_i)-1}{2}\rceil + \sum_{\epsilon_i=0} r_i\}
$$
where $d_i = k_i q+r_i$, $k_i\ge 1$, $0\le r_i \le q-1$, and the minimum is taken over all $q=p^e$ and all $(\epsilon_1, \epsilon_2, \epsilon_3)\in \{0, 1\}^3$.
\end{theorem}
Note that if the minimum is achieved for $q=p^0=1$, then $\mathcal{E}_p(d_1, d_2, d_3)=\mathcal{E}_0(d_1, d_2, d_3)$. Unlike the case $n\ge 3$, in the case $n=2$ the minimum can be achieved for any power of $p$, not just the largest which is less than $d_1, d_2, d_3$. Also note that, unlike Theorem~\ref{main_theorem}), the statement of Theorem~\ref{Han} can be applied to arbitrary values $d_1, d_2, d_3$ (since here it is not required that  $k_i <p$).
\begin{proof}
We will only show how our statement follows from the statement of Theorem 2.25 in \cite{Han_thesis}.
The inequality ($\le $) follows from Theorem (\ref{inequality}) and Theorem (\ref{ch0}). We only show ($\ge$).
According to \cite{Han_thesis}, let $q=p^m$ be the largest power of $p$ such that $\displaystyle (\frac{d_1}{q}, \frac{d_2}{q}, \frac{d_3}{q})$ is octahedral  (in the terminology of \cite{Han_thesis}, a point in the plane is octahedral if it belongs to an open unit ball with the center at $(x_1, x_2, x_3)$, where $x_1, x_2, x_3$ are non-negative integers such that $x_1+x_2+x_3$ is odd; the distance used in this definition is the taxi-cab distance, $d^*(P, Q)=\sum_{i=1}^3 |P_i -Q_i|$).
Theorem 2.25 in \cite{Han_thesis} asserts that 
$$
\mathcal{E}_p(d_1, d_2, d_3)=\frac{d_1+d_2+d_3}{2}-\frac{1}{2}d^*((d_1, d_2, d_3), qF)
$$
where $F$ is the union of the planes of equation $a_1x_1+a_2x_2+a_3x_3=b$ with $a_1, a_2, a_3 \in \{1, -1\}$, and $b=$ even integer.
Note that we must have $(x_1, x_2, x_3)=(k_1+\epsilon_1, k_2+\epsilon_2, k_3+\epsilon_3)$ for some choice of $(\epsilon_1, \epsilon_2, \epsilon_3) \in \{0, 1\}^3$.
We want to find a point $(z_1, z_2, z_3) \in F$ such that
\begin{equation}\label{want}
\sum_{i=1}^3 |qz_i - d_i| \le \sum_{i=1}^3 d_i -2 q\lceil \frac{\sum_{i=1}^3 (k_i + \epsilon_i)-1}{2}\rceil - 2(\sum_{\epsilon_i=0}r_i).
\end{equation}
This will show that $q \lceil \frac{\sum_{i=1}^3 (k_i + \epsilon_i)-1}{2}\rceil + \sum_{\epsilon_i =0} r_i \le \mathcal{E}_p(d_1, d_2, d_3)$ for this choice of $\epsilon_1, \epsilon_2, \epsilon_3$, and the conclusion will follow.

We will discuss four cases, according to the possible values for $\sum_{i=1}^3 \epsilon_i \in \{0, 1, 2, 3\}$.

If $\sum_{i=1}^3 \epsilon_i=0$, then in order for $\displaystyle (\frac{d_1}{q},\frac{d_2}{q}, \frac{d_3}{q})$ to belong to the open unit ball with center at $(k_1, k_2, k_3)$, we must have $r_1 + r_2 + r_3 < q$, and $k_1+ k_2+k_3$ must be odd.  In this case, the right hand side of equation (\ref{want}) is $q-(r_1+r_2+r_3)$. Since $r_1+r_2+r_3<q$, we can choose $c_1, c_2, c_3 >0$ real numbers such that $c_1+c_2+c_3=1$ and $qc_i \ge  r_i$ for all $i \in \{1, 2, 3\}$. Then we have $(z_1, z_2, z_3):=(k_1+c_1, k_2+c_2, k_3+c_3) \in F$ (satisfying the equation $z_1+z_2+z_3=k_1+k_2+k_3+1=$ an even integer) and $\sum_{i=1}^3 |qz_i -d_i|=\sum_{i=1}^3 (qc_i -r_i)=q-(r_1+r_2+r_3)$ and thus equation (\ref{want}) holds.

If $\sum_{i=1}^3 \epsilon_i=1$ we may assume with no loss of generality that $\epsilon_1=1$ and $\epsilon_2=\epsilon_3=0$. Then $k_1+k_2+k_3$ must be even, and in order for $\displaystyle (\frac{d_1}{q},\frac{ d_2}{q}, \frac{d_3}{q})$ to belong to the open unit ball with center at $(k_1+1, k_2, k_3)$ we must have $(q-r_1)+r_2+r_3<q$, i.e. $r_2+r_3<r_1$.
The right hand side of equation (\ref{want}) is $r_1-(r_2+r_3)$. We choose $c_1, c_2, c_3>0$ real numbers such that $c_1=c_2+c_3$, $c_1q\le r_1$,  and $c_iq\ge r_i$ for each $i \in \{2, 3\}$. We let $(z_1, z_2, z_3):=(k_1+c_1, k_2+c_2, k_3+c_3)\in F$ (satisfying the equation $z_1-z_2-z_3=k_1-k_2-k_3=$ an even integer), and we see that
$\sum_{i=1}^3 |qz_i-d_i|= (r_1-c_1q)+(c_2q-r_2)+(c_3q-r_3)=r_1-r_2-r_3$ and therefore equation (\ref{want}) holds.

If $\sum_{i=1}^3 \epsilon_i=2$, we may assume with no loss of generality that $\epsilon_1=\epsilon_2=1$ and $\epsilon_3=0$. Then $k_1+k_2+k_3$ must be odd, and in order for $\displaystyle (\frac{d_1}{q}, \frac{d_2}{q}, \frac{d_3}{q})$ to be in the open unit ball with center at $(k_1+1, k_2+1, k_3)$ we must have $(q-r_1)+(q-r_2)+r_3< q$, i.e. $q+r_3<r_1+r_2$. The right hand side of equation (\ref{want}) is $r_1+r_2-q-r_3$.  Let $c_1, c_2, c_3$ be such that $c_iq\le r_i$ for $i\in \{1, 2\}$, $c_3q\ge r_3$, and $c_1+c_2=c_3+1$. Let $(z_1, z_2, z_3):=(k_1+c_1, k_2+c_2, k_3+c_3) \in F$ (satisfying the equation $z_1+z_2-z_3=  k_1+k_2-k_3+1=$ an even integer) and $\sum_{i=1}^3 |qz_i-d_i|=(r_1-c_1q)+(r_2-c_2q)+(c_3q-r_3)=r_1+r_2-r_3-q$ and therefore equation (\ref{want}) holds.

If $\sum_{i=1}^3 \epsilon_i=3$ we have $\epsilon_1=\epsilon_2=\epsilon_3=1$ . Then $k_1+k_2+k_3$ is even, and in order for $\displaystyle (\frac{d_1}{q}, \frac{d_2}{q}, \frac{d_3}{q})$ to be in the open unit ball with center at $(k_1+1, k_2+1, k_3+1)$ we must have $(q-r_1)+(q-r_2)+(q-r_3)<q$, i.e. $r_1+r_2+r_3>2q$. The right hand side of equation (\ref{want}) is $r_1+r_2+r_3-2q$.
Choose $c_1, c_2, c_3$ such that $c_1+c_2+c_3=2$ and $c_i q < r_i$ for each $i \in \{1, 2, 3\}$. Let $(z_1, z_2, z_3)=(k_1+c_1, k_2+c_2, k_3+c_3) \in F$ (satisfying the equation $z_1+z_2+z_3=k_1+k_2+k_3+2=$ an even integer). We have
$\sum_{i=1}^3 |qz_i -d_i|=\sum_{i=1}^3 (r_i -c_iq) = r_1 +r_2+r_3 -2q$, and therefore equation (\ref{want}) holds.

\end{proof}

\section{Proof of the second inequality}
We now prepare to give the proof of the inequality ($\ge$) in Theorem (\ref{main_theorem}). The proof will be by induction on $e$ (where $q=p^e$). The base case $e=0$ is contained in Lemma (\ref{char0}). We will use the following notation:
\begin{notation}
Let $d_i =k_i q+ r_i$ for $i \in \{1, \ldots, n+1\}$ with $q=p^e$, $1\le k_i \le p-1$, $0 \le r_i \le q-1$. We let
$${\mathcal Min}(d_1, \ldots, d_{n+1}):=\mathrm{min}\{q\,  \mathcal{E}(k_1 + \epsilon_1, \ldots, k_{n+1}+\epsilon_{n+1}) + 
\sum_{\epsilon_i =0} r_i\} $$
where the minimum is taken over all $\epsilon_1, \ldots, \epsilon_{n+1} \in \{0, 1\}$.
\end{notation}
\begin{lemma}\label{inequalities}
We have
$$ \mathcal{E}_p(d_1, \ldots, d_{n+1}) \le \mathcal{E}_p(d_1+1, d_2, \ldots, d_{n+1})\le \mathcal{E}_p(d_1, d_2, \ldots, d_{n+1}) +1.$$
\end{lemma}
\begin{proof}
Let $\mathcal{A}=(a_1, \ldots, a_{n+1})^t$ be a non-Koszul relation on $x_1^{d_1}, \ldots, x_n^{d_n}, f^{d_{n+1}}$ of degree equal to $\mathcal{E}_p(d_1, d_2, \ldots, d_{n+1})$. Then $\mathcal{A'} =(a_1, x_1a_2, \ldots, x_1 a_{n+1})^t$ is a relation on $x_1^{d_1+1}, x_2^{d_2}, \ldots, x_n^{d_n}, f^{d_{n+1}}$ and $\mathrm{deg}(\mathcal{A'})=\mathrm{deg}(\mathcal{A})+1$. For both $\mathcal{A}$ and $\mathcal{A'}$, the non-Koszul property translates to $a_1 \notin (x_2^{d_2}, \ldots, x_n^{d_n}, f^{d_{n+1}})$. This proves the second inequality.

Now we prove the first inequality. Let $\mathcal{B}=(b_1, b_2, \ldots, b_{n+1})^t$ be a non-Koszul relation on $x_1^{d_1+1}, \ldots, x_n^{d_n}, f^{d_{n+1}}$, of degree $\mathcal{E}_p(d_1+1, d_2, \ldots, d_{n+1})$. Then $\mathcal{B}':=(x_1b_1, b_2, \ldots, b_{n+1})^t$ is a relation on $x_1^{d_1}, \ldots, x_n^{d_n}, f^{d_{n+1}}$ of the same degree. If $\mathcal{B'}$ is not Koszul, the inequality follows. If $\mathcal{B'}$ is Koszul, then we have $b_1x_1 \in (x_2^{d_2}, \ldots, x_n^{d_n}, f^{d_{n+1}})$. Since $\mathcal{B}$ is not Koszul, $b_1 \notin (x_2^{d_2}, \ldots, x_n^{d_n}, f^{d_{n+1}})$, and it follows that $\mathrm{deg}(b_1)+1 \ge \mathcal{E}_p(1, d_2, \ldots, d_{n+1})$. Repeated application of the second inequality (which we have already proved) yields $\mathcal{E}_p(1, d_2, \ldots, d_{n+1})\ge \mathcal{E}_p(d_1, \ldots, d_{n+1}) -(d_1-1)$, and it follows that $\mathcal{E}_p(d_1+1, d_2, \ldots, d_{n+1})=\mathrm{deg}(\mathcal{B})=d_1 +1+ \mathrm{deg}(b_1) \ge  \mathcal{E}_p(d_1, \ldots, d_{n+1})+1$ as desired.

\end{proof}

For exponents that are small compared to the characteristic, we have $\mathcal{E}_p(d_1, \ldots, d_{n+1})=\mathcal{E}_0(d_1, \ldots, d_{n+1})$.
More precisely, we have:

\begin{lemma}\label{char0}
Let $1\le k_1, \ldots, k_{n+1}\le p-1$. We have 
\begin{equation}\label{condition'}
x_i^{k_i}\notin (x_1^{k_1}, \ldots, \hat{x_i}^{k_i}, \ldots, x_{n+1}^{k_{n+1}}) \Leftrightarrow
\end{equation}
$$
k_i \le \sum_{j\ne i} (k_j -1) = \sum _{j\ne i} k_j- n \Leftrightarrow k_i < \lceil \frac{\sum_{j=1}^{n+1} k_j - n +1}{2}\rceil 
$$
 If this condition holds for all $i\in \{1, \ldots, n+1\}$, then
\begin{equation}\label{ch0eq}
\mathcal{E}_p(k_1, \ldots, k_{n+1})=\mathrm{min} \left\lbrace \lceil \frac{\sum_{i=1}^{n+1} k_i - n+1 }{2} \rceil , p \right \rbrace
\end{equation} 
\end{lemma}
We summarize the two statements in Lemma \ref{char0} in the following equation:
$$
\mathcal{E}_p(k_1, \ldots, k_{n+1})=\mathrm{max}\left\lbrace k_1, \ldots, k_{n+1}, \mathrm{min}\left\lbrace \lceil \frac{\sum_{i=1}^{n+1} k_i - n +1 }{2}\rceil, p\right\rbrace  \right\rbrace
$$
whenever $1 \le k_1, \ldots, k_{n+1} \le p$. Also note that equation (\ref{ch0eq}) continues to hold if $k_i=\sum_{j\ne i}(k_j-1)+1$ or $k_i =\sum_{j\ne i} (k_j-1) +2$, since in this case we have $\displaystyle k_i=\lceil \frac{\sum_{i=1}^{n+1} k_i - n+1}{2} \rceil$.

The proof goes along the same lines as \cite{RRR}.
\begin{proof}
The statement in equation (\ref{condition'}) is obvious because all the multinomial coefficients involved in the expansion of $x_i^{k_i}=(x_{n+1}-(x_1 + \ldots + x_{i-1} + x_{i+1} + \ldots + x_n))^{k_i}$ are nonzero and therefore the necessary and sufficient condition for $x_i^{k_i}$ to be in $(x_1^{k_1}, \ldots, \hat{x_i}^{k_i}, \ldots, x_{n+1}^{k_{n+1}})$ is that all monomials in the variables $x_1, \ldots, x_{i-1}, x_{i+1}, \ldots, x_{n+1}$ of degree $k_i$ are divisible by one of $x_1^{k_1}, \ldots, \hat{x_i}^{k_i}, \ldots, x_{n+1}^{k_{n+1}}$.

Note that 
\begin{equation}
\label{rel3}(x_1 + \ldots + x_n)^p = x_1^p + \ldots +x_n^p
\end{equation}
is a relation on $x_1^{k_1}, \ldots, x_n^{k_n}, f^{k_{n+1}}$ of total degree $p$. Moreover, (\ref{rel3}) is a Koszul relation if and only if 
$
(x_1 + \ldots + x_n)^{p-k_{n+1}} \in (x_1^{k_1}, \ldots, x_n^{k_n}).
$
Since all the multinomial coefficients in the expansion of $(x_1 + \ldots + x_n)^{p-k_{n+1}}$ are non-zero, this is equivalent to $p-k_{n+1} \ge \sum_{i=1}^n k_i - n +1$ (so that every monomial in the expansion is divisible by one of $x_1^{k_1}, \ldots, x_n^{k_n}$). So equation (\ref{rel3}) is a Koszul relation on $x_1^{k_1}, \ldots, x_n^{k_n}, f^{k_{n+1}}$  if and only if $p\ge \sum_{i=1}^{n+1}k_i-n+1$ (in which case the minimum is not $p$).
We know from Theorem (\ref{one_direction}) that $\mathcal{E}_p(k_1, \ldots, k_{n+1})\le \lceil \frac{\sum_{i=1}^{n+1} k_i -n +1}{2}\rceil$, and this proves that when $p\le \sum_{i=1}^{n+1}  k_i -n$ we also have $\mathcal{E}_p(k_1, \ldots, k_{n+1})\le p$. Thus we conclude that
$$
\mathcal{E}(k_1, \ldots, k_{n+1})\le \mathrm{min}\left\lbrace \lceil \frac{\sum_{i=1}^{n+1}k_i -n +1 }{2} \rceil, p \right\rbrace.
$$

For the other inequality, consider a homogeneous non-Koszul relation 
\begin{equation}\label{rel4}
g(x_1+ \ldots + x_n)^{k_{n+1}} =a_1x_1^{k_1}+ \ldots + a_nx_n^{k_n}
\end{equation}
We want to show that $\mathrm{deg}(g) +k_{n+1} \ge \mathrm{min}\left\lbrace \lceil \frac{\sum_{i=1}^{n+1}k_i -n +1 }{2} \rceil, p \right\rbrace$.
We use induction on $\mathrm{deg}(g)$.

If $\mathrm{deg}(g)=0$, we must have $k_{n+1}\ge p$, or $k_{n+1} \ge k_1 + \ldots k_n -n +1$ (from the statement in equation (\ref{condition'})), and the desired inequality holds.

Assume $\mathrm{deg}(g)>0$ and let $i$ be such that $\partial g/\partial x_i \ne 0$ (note that we can assume that $\mathrm{deg}(g) <p$ and therefore $g$ is not a $p$th power, and such a $i$ exists). Take the derivatives with respect to $x_i$ of both sides in equation (\ref{rel4}).
We get
$$
\frac{\partial g} {\partial x_i} f^{k_{n+1}} + gk_{n+1}f^{k_{n+1}-1} \in (x_1^{k_1}, \ldots, x_i^{k_i-1}, \ldots, x_n^{k_n})
$$
Multiply through by $f$ and use equation (\ref{rel4}); we obtain 
\begin{equation}\label{rel5}
\frac{\partial g}{\partial x_i} f^{k_{n+1} + 1} \in (x_1^{k_1}, \ldots, x_i^{k_i-1}, \ldots, x_n^{k_n})
\end{equation}
This is a non-Koszul relation in $x_1^{k_1}, \ldots, x_i^{k_i-1}, \ldots, x_n^{k_n}, f^{k_{n+1}+1}$ (non-Koszul because $\partial g/\partial x_i \in (x_1^{k_1}, \ldots, x_i^{k_i-1}, \ldots, x_n^{k_n})$ implies $g \in (x_1^{k_1}, \ldots,x_i^{k_i}, \ldots,  x_n^{k_n})$ by integration, and we know that this is not the case by the assumption that relation (\ref{rel4}) is non-Koszul). The conclusion follows by applying the induction hypothesis to (\ref{rel5}).
\end{proof}

\begin{lemma}\label{less_than_q}
Assume that $d_1<q$ and $d_{n+1} =k_{n+1}q+r_{n+1}\ge q$ where $q=p^e$ is a power of $q$, $k_{n+1}\ge 1$ and $0 \le r_{n+1} \le q-1$. Then $$\mathcal{E}_p(d_1, \ldots, d_{n+1})\ge \mathrm{min}\{\mathcal{E}_p(q, d_2, \ldots, d_n, k_{n+1}q), k_{n+1}q+\mathcal{E}_p(d_1, \ldots, d_n, r_{n+1})\}.$$
\end{lemma}
\begin{proof}
Note that $f^{k_{n+1}q}\equiv (x_2+ \ldots + x_n)^{k_{n+1}q}$ (mod $(x_1^q)$).

Consider a non-Koszul relation $\mathcal{A}=(a_1, \ldots, a_{n+1})^t$ on $x_1^{d_1}, \ldots, x_n^{d_n}, f^{d_{n+1}}$.
Then
$$a_1 \in (x_2^{d_2}, \ldots, x_n^{d_n}, x_1^q, (x_2+ \ldots + x_n)^{k_{n+1}q}):x_1^{d_1} = 
$$
$$
(x_2^{d_2}, \ldots, x_n^{d_n}, x_1^{q-d_1}, (x_2+ \ldots + x_n)^{k_{n+1}q})= (x_2^{d_2}, \ldots, x_n^{d_n}, x_1^{q-d_1}, f^{k_{n+1}q})$$
(here we used the fact that $x_1, \ldots, x_n$ form a regular sequence).

Write $a_1 = b_1x_1^{q-d_1} +b_2x_2^{d_2} + \ldots + b_nx_n^{d_n}+ b_{n+1}f^{k_{n+1}q}$ and substitute back in the original relation. 

It follows that $\mathcal{B}:=(b_1, a_2', \ldots, a_n', b_{n+1}x_1^{d_1}+a_{n+1}f^{r_{n+1}})^t$ is a relation on $x_1^q, x_2^{d_2}, \ldots, x_n^{d_n}, f^{k_{n+1}q}$, where $a_i'=a_i + b_i x_1^{d_1}$.  Note that $\mathrm{deg}(\mathcal{B})=\mathrm{deg}(\mathcal{A})$. Therefore we have either $\mathrm{deg}(\mathcal{A}) \ge \mathcal{E}_p(q, d_2, \ldots, d_n, k_{n+1}q)$ (which gives the desired conclusion), or $\mathcal{B}$ is a Koszul relation. In the latter case, we have $b_{n+1}x_1^{d_1}+a_{n+1}f^{r_{n+1}}\in (x_1^q, x_2^{d_2}, \ldots, x_n^{d_n})$, thus $a_{n+1}f^{r_{n+1}}\in (x_1^{d_1}, \ldots, x_n^{d_n})$, and since $\mathcal{A}$ is a non-Koszul relation, it follows that $\mathrm{deg}(a_{n+1}) + r_{n+1} \ge \mathcal{E}_p(d_1, d_2, \ldots, d_n, r_{n+1})$, and therefore $\mathrm{deg}(\mathcal{A}) =\mathrm{deg}(a_{n+1})+k_{n+1}q + r_{n+1} \ge k_{n+1}q+ \mathcal{E}_p(d_1, d_2, \ldots, d_n, r_{n+1})$.

\end{proof}

\begin{definition}
If $d_i\ge d_i'$ for all $i=1, \ldots, n+1$, we say that a relation ${\mathcal A}=(a_1, \ldots, a_{n+1})^t$ on $x_1^{d_1}, \ldots, x_n^{d_n}, f^{d_{n+1}}$ {\it restricts to} the relation ${\mathcal A'}:=(x_1^{d_1-d_1'}a_1, \ldots, f^{d_{n+1}-d_{n+1}'}a_{n+1})^t$ on $x_1^{d_1'}, \ldots, x_n^{d_n'}, f^{d_{n+1}'}$. Note that $\mathrm{deg}({\mathcal A'})=\mathrm{deg}({\mathcal A})$.
\end{definition}

The next two Lemmas prove the result for the case when $r_i=0$ for some $i$ (since $\mathcal{E}_p(d_1, \ldots, d_{n+1})$ is symmetric in the $d_i$'s,  we may assume $i=n+1$ without loss of generality). We note that in this case the assumption that $k_i \le \sum_{j \ne i} k_j-n +1$ for all $i\in \{1, \ldots, n\}$ is not needed. Lemma (\ref{Lem}) below states that when $r_{n+1}=0$, then all the relations can be constructed as in the proof of Theorem (\ref{inequality}) and thus we not only know the smallest degree of a relation, but we can explicitely describe all the relations. This will no longer be true when $r_i>0$ for all $i \in \{1, \ldots, n+1\}$.

\begin{lemma}\label{Lem}
If $d_i=k_iq+r_i$ with $q=p^e$, $k_i\ge 1$ and $0\le r_i\le q-1$ for all $i=1, \ldots, n$, then every relation on $x_1^{d_1}, \ldots, x_n^{d_n}, f^{k_{n+1}q}$ restricts to a relation on $x_1^{k_1q}, \ldots, x_n^{k_nq}, f^{k_{n+1}q}$
of the form 
\begin{equation}\label{sum}
\sum_M g_{M, l} M \left(\begin{array}{c} a_{1, M, l}^q \\ \vdots \\ a_{n+1, M, l}^q\\ \end{array}\right)
\end{equation}
where $M$ ranges through all monomials of the form $M=x_{i_1}^{r_{i_1}}\cdots x_{i_s}^{r_{i_s}}$ where $1 \le i_1 < \ldots < i_s \le n$, $l$ runs through some indexing set,  $g_{M, l}\in k[x_1, \ldots, x_n]$, and for each $M, l$, $(a_{1, M, l}, \ldots, a_{n+1, M, l})^t$ is a relation on $x_1^{k_1}, \ldots, x_n^{k_n}, f^{k_{n+1}}$ which is restricted from a relation on $x_1^{k_1+\epsilon_1}, \ldots, x_n^{k_n+\epsilon_n}, f^{k_{n+1}}$ where $\epsilon _i=0$ for $i \in \{i_1, \ldots, i_s\}$ and $\epsilon _i =1$ for $i \notin \{ i_1, \ldots, i_s\}$.

Furthermore, if we work modulo the Koszul relations on $x_1^{d_1}, \ldots, x_n^{d_n}, f^{k_{n+1}q}$, then we may assume that each relation $(a_{1, M, l}, \ldots, a_{n+1, M, l})^t$ is the sum (\ref{sum}) is restricted from a non-Koszul relation on $x_1^{k_1+ \epsilon_1}, \ldots, x_n^{k_n + \epsilon_n}, f^{k_{n+1}}$.
\end{lemma}

\begin{proof}
We know by the flatness of the Frobenius functor that any relation on $x_1^{k_1q}, \ldots, x_n^{k_nq}, f^{k_{n+1}q}$ is a linear combination of relations of the form $(a_{1, l}^q, \ldots, a_{n+1, l}^q)^t$, where $(a_{1, l}, \ldots, a_{n+1, l})^t$ are relations on $x_1^{k_1}, \ldots, x_n^{k_n}, f^{k_{n+1}}$. 
We view $k[x_1, \ldots, x_n]$ as a free module over the ring $k[x_1^q, \ldots, x_n^q]$, with basis consisting of all the monomials $x_1^{j_1}\cdots x_n^{j_n}$ with $0 \le j_1, \ldots, j_n \le q-1$.
Therefore any relation on  
$x_1^{k_1q}, \ldots, x_n^{k_nq}, f^{k_{n+1}q}$
can be written as
$$
\sum _{\mu} {\mu}\left(\begin{array}{c} a_{1, \mu}^q \\ \vdots \\ a_{n+1, \mu}^q\\ \end{array}\right)
$$
where $\mu$ ranges through all the monomials $x_1^{j_1}\cdots x_n^{j_n}$ with $0 \le j_1, \ldots, j_n \le q-1$ (we are using the assumption that $k$ is a perfect field in order to incorporate scalar coefficients as part of the $a_{i, {\mu}}^q)$. Moreover, in order for the relation to be restricted from a relation on $x_1^{d_1}, \ldots, x_n^{d_n}, f^{k_{n+1}q}$ we need the entry in the $l$th component, namely $\sum _{\mu} \mu a_{l, \mu}^q$ to be divisible by $x_l^{r_l}$ for all $l=1, \ldots, n$. This implies that for every monomial $\mu=x_1^{j_1}\cdots x_n^{j_n}$ that has $j_l<r_l$, we must have $a_{l, \mu}$ divisible by  $x_l$ (one can see that the result of multiplying $x_l^{r_l}$ by any polynomial written in terms of the free basis of $k[x_1, \ldots, x_n]$ over $k[x_1^q, \ldots, x_n^q]$ is of the form $\sum G_{\mu} \mu $ where either $\mu$ is divisible by $x_l^{r_l}$, or $G_{\mu}$ is divisible by $x_l^q$). In other words, the relation $(a_{1, \mu}, \ldots, a_{n+1, \mu})^t$ where $\mu = x_1^{j_1} \cdots x_n^{j_n}$ is restricted from a relation on $x_1^{k_1+\epsilon_1}, \ldots, x_n^{k_n + \epsilon_n}, f^{k_{n+1}}$, where $\epsilon_l =1$ when $j_l< r_l$ and $\epsilon_l =0$ otherwise.

The claim now follows by combining all the monomials divisible by $x_{i_1}^{r_{i_1}}\cdots x_{i_s}^{r_{i_s}}$ but not by any $x_j^{r_j}$ for $j\notin \{i_1, \ldots, i_s\}$ into terms of the form $g_{M, l} M$ with $M=x_{i_1}^{r_{i_1}}\cdots x_{i_s}^{r_{i_s}}$, giving equation (\ref{sum}). Each $a_{j, M, l}$ is a $a_{j, \mu}$ where $\mu$ ranges through the monomials described above, and therefore $a_{j, M, l}$ is divisible by $x_j$ for all $j \notin \{i_1, \ldots, i_s\}$. This implies that $(a_{1, M, l}, \ldots, a_{n+1, mM, l})^t$ is restricted from a relation on $x_1^{k_1+\epsilon_1}, \ldots, x_n^{k_n+\epsilon_n}, f^{k_{n+1}}$ where $\epsilon _i=0$ for $i \in \{i_1, \ldots, i_s\}$ and $\epsilon _i =1$ for $i \notin \{ i_1, \ldots, i_s\}$.

In order to justify the last paragraph in the statement, note that if $a_{n+1, \mu} \in (x_1^{k_1+ \epsilon_1}, \ldots, x_n^{k_n + \epsilon_n})$, then it follows that $\mu a_{n+1, \mu}^q \in (x_1^{k_1q+r_1}, \ldots, x_n^{k_n q + r_n})$, where $\mu= x_1^{j_1} \cdots x_n^{j_n}$ is such that $j_l \ge r_l$ whenever $\epsilon_l =0$.

\end{proof}

\begin{corollary}\label{Cor}
 With notation as above, we have
$$\mathcal{E}_p(k_1q+r_1, \ldots, k_nq+r_n,  k_{n+1}q)=
{\mathcal Min}(k_1q+r_1, \ldots  k_nq+r_n,  k_{n+1}q).$$

Moreover, if $\mathcal{A}$ is a non-Koszul relation on $x_1^{k_1q+r_1}, \ldots, x_n^{k_nq+r_n}, f^{k_{n+1}q+r_{n+1}}$ such that the restriction of $\mathcal{A}$ to $x_1^{k_1q+r_1}, \ldots, x_n^{k_nq+r_n}, f^{k_{n+1}q}$ is Koszul, then we have $\mathrm{deg}({\mathcal A}) \ge {\mathcal Min}(k_1q+r_1, \ldots, k_nq+r_n, k_{n+1}q+r_{n+1})$.
\end{corollary}
\begin{proof}

 The first statement follows immediately from Lemma (\ref{Lem}) by observing that if any $(a_{1, M, l}, \ldots, a_{n+1, M, l})^t$ from equation (\ref{sum}) is a Koszul relation on $x_1^{k_1+\epsilon_1}, \ldots, x_n^{k_n+\epsilon_n}, f^{k_{n+1}}$, then $a_{n+1, M, l}\in (x_1^{k_1+\epsilon_1}, \ldots, x_n^{k_n+\epsilon_n})$, which implies $M a^q_{n+1, M, l} \in (x_1^{k_1q+r_1}, \ldots, x_n^{k_nq+r_n})$ and therefore $M (a^q_{1, M, l}, \ldots, a^q_{n+1, M, l})^t$ is a Koszul relation. Thus we may assume without loss of generality that all the relations on the right hand side (\ref{sum}) are non-Koszul.

Now we prove the second statement.
The assumption that the restriction of $\mathcal{A}$ is Koszul but $\mathcal{A}$ is not implies that $a_{n+1} f^{r_{n+1}} \in (x_1^{k_1q+r_1}, \ldots, x_n^{k_nq+r_n})$ but $a_{n+1} \notin (x_1^{k_1q+r_1}, \ldots, x_n^{k_nq+r_n})$. Using the second inequality from  Lemma (\ref{inequalities}),  we have $$\mathrm{deg}(a_{n+1}) + r_{n+1} \ge \mathcal{E}_p(k_1q+r_1, \ldots, k_nq+r_n, r_{n+1}) \ge 
$$
$$ \mathcal{E}_p(k_1q+r_1, \ldots, k_nq+r_n, q)-q+r_{n+1}.$$ From the first part of the statement, we now know that 
$$\mathrm{deg}(a_{n+1}) + r_{n+1} \ge {\mathcal Min}(d_1, \ldots, d_n, q) -q + r_{n+1}=
$$
$$
\mathrm{min}\{\mathcal{E}_p(k_1+ \epsilon_1, \ldots, k_n+\epsilon_n, 1) q +  \sum_{\epsilon_j=0, j \le n } r_j\}-q+r_{n+1},$$
 and therefore
$$\mathrm{deg}(\mathcal{A}) = k_{n+1} q+r_{n+1} + \mathrm{deg}(a_{n+1}) \ge
$$
$$
 \mathrm{min}\{\mathcal{E}_p(k_1+ \epsilon_1, \ldots, k_n+\epsilon_n, 1) q + \sum_{\epsilon_j=0, j \le n} r_j\} +(k_{n+1}-1)q+  r_{n+1}\ge
$$
$$ \mathrm{min}\{\mathcal{E}_p(k_1+ \epsilon_1, \ldots, k_n+\epsilon_n, k_{n+1}) q + \sum_{\epsilon_j=0, j \le n} r_j\}+ r_{n+1}$$where the last inequality follows from Lemma (\ref{inequalities}). Since the right hand side of the inequality above is greater than or equal to ${\mathcal Min}(d_1, \ldots, d_{n+1})$, we have the desired conclusion.
\end{proof}

\begin{lemma}\label{remainders}
Let $p \ge 3$, $q=p^e$ with $e \ge 1$, $n\ge 3$ and assume by the inductive hypothesis that the conclusion of Theorem (\ref{main_theorem}) holds for $q'=p^{e-1}$.

Let $1 \le r_1 \le r_2 \le \ldots \le  r_{n+1} \le q-1$ be such that $r_n + r_{n+1} \ge q$. Then
$$
\mathcal{E}_p(r_1, \ldots, r_{n+1}) \ge \mathrm{min}( r_1+r_2, q)
$$

\end{lemma}
Before giving the proof, we wish to emphasize the fact that the assumption that $n \ge 3$ is essential in Lemma (\ref{remainders}), and that this is the main reason for the difference in the main result for $n\ge 3$ compared to the result in (\cite{Han_thesis}) for $n=2$.

\begin{proof}
Note that for any positive integers $1\le l_1\le l_2\ldots\le  l_{n+1}$, we have 
\begin{equation}\label{this_is_why}
\lceil \frac{\sum_{i=1}^{n+1} l_i - n+1 }{2} \rceil  \ge l_1 + l_2 -1
\end{equation}
and the inequality is strict unless $n=3$ and $l_1=\ldots = l_4$, or $n\ge 4$ and $l_1=\ldots = l_{n+1}=1$. This is because $\sum_{i=1}^{n+1} l_i -n +1 = \sum_{i=1}^{n-1} (l_i -1) + l_n + l_{n+1} \ge l_1-1 + l_2 -1 + l_n + l_{n+1} \ge 2l_1 + 2l_2 -2$.

Consider the case $q=p$. Then the desired conclusion follows from Lemma (\ref{char0}) together with the inequality  (\ref{this_is_why}), since in the case $r_1=\ldots = r_{n+1}=r$ the assumption that $r_n+r_{n+1}\ge p$ guarantees that $2r-1\ge p$ (since $p$ is odd). 

Now let $q=p^e$ with $e \ge 2$. Write $r_i = u_i q' + v_i$ where $q'=p^{e-1}$, $0 \le u_i \le p-1$, $0 \le v_i \le q'-1$ and we are assuming that the conclusion of Theorem (\ref{main_theorem}) holds for $q'=p^{e-1}$ instead of $q$.

We first treat the case when Theorem (\ref{main_theorem}) can be applied to $r_1, \ldots, r_{n+1}$, i.e. $1\le u_1 \le u_2 \le \ldots \le u_{n+1} \le \sum_{i=1}^n u_i -n +1$. Then we have
$$\mathcal{E}_p(r_1, \ldots, r_{n+1})=\mathrm{min}\{q'\mathcal{E}_p(u_1 + \epsilon_1, \ldots, u_{n+1}+\epsilon_{n+1})+ \sum_{\epsilon_i =0} v_i\}$$
 with the minimum taken over  all the choices of $\epsilon_1, \ldots, \epsilon_{n+1} \in \{0, 1\}$. Fix $\epsilon_1, \ldots, \epsilon_{n+1}$ for which the minimum is reached. If $\mathcal{E}(u_1+\epsilon_1, \ldots, u_{n+1}+\epsilon_{n+1})=p$, then we have $\mathcal{E}(r_1, \ldots, r_{n+1})\ge pq'=q$ and we have the desired conclusion. Thus we will assume that
$$
\mathcal{E}(u_1+ \epsilon_1, \ldots, u_{n+1}+\epsilon_{n+1})=\lceil \frac{\sum _{i=1}^{n+1} (u_i+\epsilon_i) - n+1}{2} \rceil.
$$
We consider the following cases: $\epsilon_1+\epsilon_2=0, \epsilon_1+\epsilon_2=1, \epsilon_1+\epsilon_2=2$.

{\bf Case 1:} $\epsilon_1+\epsilon_2=0$. 
The inequality (\ref{this_is_why}) implies that $\mathcal{E}_p(u_1, u_2, u_3+\epsilon_3, \ldots, u_{n+1}+\epsilon_{n+1}) \ge u_1 + u_2 -1$, with strict inequality unless $u_1=u_2 =u_3+\epsilon_3=\ldots = u_{n+1} + \epsilon_{n+1}:=u$. If the inequality (\ref{this_is_why})  is strict, then we have 
$\mathcal{E}_p(r_1, \ldots, r_{n+1}) \ge (u_1 + u_2) q' + v_ 1 + v_2 = r_1 +r_2$ and thus the desired conclusion follows. If the inequality (\ref{this_is_why}) is an equality, then we must have $\epsilon_3 =\ldots = \epsilon_{n+1}=0$ and $u_1 = \ldots = u_{n+1}=u$.

If $2u-1\ge p$, then we have $\mathcal{E}_p(r_1, \ldots, r_{n+1}) \ge pq'=q$, and the desired conclusion holds. Otherwise, since $p$ is odd, we must have $2u-1 \le p-2$.
 Recalling the assumption that $r_n+r_{n+1} =2uq'+v_n+v_{n+1}\ge q$, it follows that $v_n + v_{n+1} \ge q'$, and therefore 
$\mathcal{E}_p(r_1, \ldots, r_{n+1})=(2u-1)q'+ v_1 + v_2 + \ldots + v_n + v_{n+1}  \ge 2uq'+v_1+v_2 =r_1+r_2$.

\smallskip

{\bf Case 2:} $\epsilon_1+ \epsilon_2=1$. Say that $\epsilon_1=1, \epsilon_2=0$ (the other case is similar). We first consider the case when $u_1+1=u_2=u_3+\epsilon_3=\ldots = u_{n+1}+\epsilon_{n+1}:=u$ and $\mathcal{E}(u_1+1, u_2, \ldots, u_{n+1}+\epsilon_{n+1})=2u-1$. We must have $\epsilon_3=\ldots = \epsilon_{n+1}=0$ and $u_2=\ldots = u_{n+1}=u$, $u_1=u-1$. 
If $2u-1\ge p$ then we have $\mathcal{E}_p(r_1, \ldots, r_{n+1}) \ge pq' = q$, and the desired conclusion holds. Otherwise, since $p$ is odd, we must have $2u-1\le p-2$. 
Recalling the assumption that $r_n + r_{n+1} \ge q$, it follows that $v_n + v_{n+1} \ge q'$. Thus we have
$$
\mathcal{E}_p(r_1, \ldots, r_{n+1}) = (2u-1)q' + v_2 + \ldots + v_n + v_{n+1} \ge 2uq' + v_2 \ge r_1 + r_2
$$
and the desired conclusion holds.
 
Now consider the case when inequality (\ref{this_is_why}) is strict when applied to $u_1+1, u_2, u_3+\epsilon_3, \ldots, u_{n+1}+\epsilon_{n+1}$, and therefore we have $\mathcal{E}(u_1+1, u_2, u_3+\epsilon_3, \ldots, u_{n+1} + \epsilon_{n+1}) \ge \mathrm{min} \{ u_i +\epsilon_i + u_j + \epsilon_j\}$.
If $u_1+1, u_2, u_3+\epsilon_3, \ldots, u_{n+1}+\epsilon_{n+1}$ are not all equal, then the smallest two are either $u_1+1, u_2$, or $u_2, u_j$ (where $u_j = u_1$ and $\epsilon_j =0$). If the smallest two are $u_1+1, u_2$ then we have
$\mathcal{E}_p(r_1, \ldots, r_{n+1})\ge (u_1+u_2+1)q' + v_2 \ge r_1 + r_2$. If the smallest two are $u_2, u_j$ (with $j$ as above) then 
$\mathcal{E}_p(r_1, \ldots, r_{n+1})\ge (u_2 + u_j)q' + v_2 + v_j$. But we have $u_j = u_1$ and $r_j \ge r_1$, and therefore $v_j \ge v_1$, which leads to the desired conclusion.

\smallskip

{\bf Case 3:} $\epsilon_1 = \epsilon_2 =1$. Then $\mathcal{E}_p(r_1, \ldots, r_{n+1})=\mathcal{E}_p(u_1+1, u_2 + 1, u_3 +\epsilon_3, \ldots, u_{n+1} + \epsilon_{n+1})q' + \sum_{i\ge 3, \epsilon_i = 0} v_i$.

If $u_1 + 1 = u_2 + 1 = u_3 + \epsilon_3 = \ldots = u_{n+1} + \epsilon_{n+1}=u$, then $\mathcal{E}_p(u_1+1, u_2 + 1, u_3 +\epsilon_3, \ldots, u_{n+1} + \epsilon_{n+1}) \ge 2u-1$. If $2u-1 \ge p$, it follows that $\mathcal{E}_p(r_1, \ldots, r_{n+1}) \ge q$, and we are done. Assume that $2u-1 \le p-1$. Since $p$ is odd, this means that $2u-1 \le p-2$. 
Note that this implies that $\epsilon_n =\epsilon_{n+1}=0$, since otherwise we would have $u_n+u_{n+1}\le 2u-1$, and $r_n + r_{n+1}\le (2u-1) q' + v_n + v_{n+1} < (p-2)q' + 2q' = q$, which contradicts our assumption.
Now we have $u_n =u_{n+1}=u$, and the assumption that $r_n + r_{n+1} \ge q$ implies that $2uq' + v_n + v_{n+1} \ge pq'$. Since $2u \le p-1$, it follows that $v_n + v_{n+1} \ge q'$, and we have
$\mathcal{E}_p(r_1, \ldots, r_{n+1})\ge (2u-1) q' + v_n + v_{n+1} \ge 2uq'= (u_1 + 1) q' + (u_2 + 1) q' \ge r_1 + r_2$.

If $u_1+1, u_2+1, u_3 + \epsilon_3, \ldots, u_{n+1} + \epsilon_{n+1}$ are not all equal, then the smallest two could be $u_1+1, u_2+1$, or $u_1 +1, u_j$ (where $u_j=u_2$ and $\epsilon_j =0$), or $u_j, u_l$ (where $u_j = u_1$, $u_l = u_2$, and $\epsilon_j = \epsilon_l =0$).
In the first case, it follows that $\mathcal{E}_p(r_1, \ldots, r_{n+1}) \ge (u_1 +1 + u_2 +1) q' \ge r_1 + r_2$.
In the second case, we have $\mathcal{E}_p(r_1, \ldots, r_{n+1}) \ge (u_1 + 1 + u_j) q' + v_j$, and since $u_j = u_2$ and $r_j \ge r_2$, it follows that $v_j \ge v_2$, which leads to $\mathcal{E}_p(r_1, \ldots, r_{n+1})\ge (u_1 +1) q' + u_2q' + v_ 2 \ge r_1 + r_2$.
In the third case, we have $\mathcal{E}_p(r_1, \ldots, r_{n+1}) \ge (u_j + u_l) q' + v_j + v_l$ and we must have $v_j \ge v_1, v_l \ge v_2$ for the same reason as above, implying that $\mathcal{E}_p(r_1, \ldots, r_{n+1})\ge u_1 q' + u_2 q' + v_1 + v_2 = r_1 + r_2$.

\medskip

Now we will consider all the remaining cases, when the result of Theorem (\ref{main_theorem}) cannot be applied to $r_1, \ldots, r_{n+1}$. This is the case when $u_{n+1} \ge \sum_{i=1}^n u_i  - n +2$, or when $r_i < q'$ for some $i$ (and thus $u_i=0$). The idea for the cases when all $u_i \ge 1$ is to use Lemma (\ref{inequalities}) to replace $u_{n+1}q'+v_{n+1}$ by $u_{n+1}q'$, and then apply Corollary (\ref{Cor}). By Lemma (\ref{inequalities}), we have
\begin{equation}\label{case}
\mathcal{E}_p(r_1, \ldots, r_{n+1}) \ge \mathcal{E}_p(r_1, \ldots, r_n, u_{n+1}q') = \mathcal{E}_p(u_1+ \epsilon_1, \ldots, u_n + \epsilon_n, u_{n+1}) q' + \sum_{i\le n, \epsilon_i=0} v_i
\end{equation}
for some choice of $\epsilon_1, \ldots, \epsilon_n$. 

We distinguish four cases.

{\bf Case 1:}  $u_{n+1} = \sum_{i=1}^n u_i -n +2$ and all $u_i \ge 1$. Note that this implies that $\mathcal{E}(u_1+ \epsilon_1, \ldots, u_{n+1})= u_{n+1}$ if $\sum_{i=1}^{n+1} \epsilon_i \le 1$, and $\mathcal{E}(u_1+ \epsilon_1, \ldots, u_{n+1}+\epsilon_{n+1}) \ge u_{n+1} +1$ when $\sum_{i=1}^{n+1}\epsilon_i \ge 2$.  Also note that we have $u_{n+1} \ge u_1 + u_2$.

We distinguish three possibilities:

{\bf Subcase 1:} $u_{n+1}=u_1+u_2$. Note that this only happens when $u_1= \ldots =u_n=1$, $u_{n+1}=2$. Since we are assuming that $r_n + r_{n+1} \ge q$, this can only happen when $p=3$ (since $q \le r_n + r_{n+1} \le 3q' + r_n + r_{n+1} < 5q'$ implies $p<5$).
If $\sum_{i=1}^n \epsilon_i \le 1$, it follows that $\sum_{i\le n, \epsilon_i=0} v_i$ from equation (\ref{case}) consists of at least two terms, and since $u_1=\ldots= u_n$, it follows that $v_1 \le v_2 \le \ldots \le v_n$, and therefore $\sum_{i\le n, \epsilon_i=0} v_i \ge v_1 + v_2$, and therefore $\mathcal{E}_p(r_1, \ldots, r_{n+1}) \ge u_{n+1}q' + v_1 + v_2 = r_1 + r_2$, as desired.

If $\sum_{i=1}^n \epsilon_i \ge 2$, then we have $\mathcal{E}_p(u_1+\epsilon_1, \ldots, u_n + \epsilon_n, u_{n+1})\ge u_{n+1} + 1=3$, and since $p=3$, it follows that $\mathcal{E}_p(r_1, \ldots, r_{n+1}) \ge 3q' = q$, as desired.

{\bf Subcase 2:} $u_{n+1} = u_1 + u_2 +1$. This happens when $u_n=2$, $u_{n-1}=\ldots =u_1=1$, or when $n=3$, $u_3=2$, and $u_1, u_2 \in \{1, 2\}$. If $\sum_{i=1}^n \epsilon_i \le 1$, then $\sum_{i\le n, \epsilon_i=0} v_i$ in equation (\ref{case}) consists of at least two terms, and at least one of them is $\ge \mathrm{min}\{v_1, v_2\}$ (when $n=3$ then one of the two terms must be $v_1$ or $v_2$; when $n \ge 4$ we must have $v_1 \le v_2 \le \ldots \le v_{n-1}$ and at least one of the two terms is not $v_n$). It follows that $\mathcal{E}_p(r_1, \ldots, r_{n+1}) \ge u_{n+1} q' + \mathrm{min}(v_1, v_2) = (u_1+u_2+1)q' + \mathrm{min}(v_1, v_2) \ge r_1 + r_2$ as desired. If $\sum_{i=1}^n \epsilon_i \ge 2$, then $\mathcal{E}_p(u_1+ \epsilon_1, \ldots, u_n + \epsilon_n, u_{n+1}) \ge u_{n+1} + 1$, and it follows that $\mathcal{E}_p(r_1, \ldots, r_{n+1}) \ge (u_1 + u_2 +2) q' \ge r_1 + r_2$, as desired.

{\bf Subcase 3:} $u_{n+1} \ge u_1 + u_2 +2$. Then it follows that $\mathcal{E}_p(r_1, \ldots, r_{n+1}) \ge u_{n+1} q' \ge (u_1+u_2+2)q' \ge r_1 + r_2$.
\smallskip

{\bf Case 2:} $u_{n+1} = \sum_{i=1}^n u_i -n +3$ and all $u_i \ge 1$. Then $u_{n+1}= u_1 + u_2 +u_3 +(u_4-1) + \ldots + (u_n-1) \ge u_1 + u_2 + 1$, with equality if and only if $u_1=\ldots = u_n=1$, $u_{n+1}=3$. If the inequality is strict, then we have $\mathcal{E}_p(r_1, \ldots, r_{n+1})\ge \mathcal{E}_p(r_1, \ldots, r_n, u_{n+1}q') \ge u_{n+1}q' \ge (u_1 + u_2 + 2) q' \ge r_1 + r_2$. Assume that $u_{n+1}=u_1+u_2+1$. If $\sum_{i=1}^n \epsilon_i \le 2$ in equation (\ref{case}), then the sum $\sum_{i\le n, \epsilon_i=0} v_i$ consists of at least one term, and we get $\mathcal{E}_p(r_1, \ldots, r_{n+1}) \ge u_{n+1}q' + \mathrm{min}_{i=1}^n (v_i)=(u_1+u_2+1)q' + \mathrm{min}(v_1, v_2) \ge r_1 + r_2$.
If $\sum_{i=1}^n \epsilon _i \ge 3$, then we have $\mathcal{E}_p(u_1+ \epsilon_1, \ldots, u_n + \epsilon_n, u_{n+1})\ge u_{n+1}+1$, and it follows that $\mathcal{E}_p(r_1, \ldots, r_{n+1}) \ge (u_1 + u_2 +2)q' \ge r_1 + r_2$.

{\bf Case 3:} Now assume that $u_{n+1} \ge \sum_{i=1}^n u_i - n+4$ and all $u_i \ge 1$. Then $u_{n+1} \ge u_1 + u_2 +2$ and we have $\mathcal{E}_p(r_1, \ldots, r_{n+1}) \ge u_{n+1}q' \ge (u_1 + u_2 + 2)q' \ge r_1 + r_2$.

{\bf Case 4:} The last remaining case is when some of the $r_i$'s are less than $q'$. The proof will be on induction on the number $I$ of indexes $i$ such that $r_i<q'$, with $I=0$ being the base case. 
If $I=1$, say that $r_1<q', r_2, \ldots, r_{n+1} \ge q'$. By Lemma (\ref{inequalities}), we have 
$\mathcal{E}_p(r_1, \ldots, r_{n+1}) \ge \mathcal{E}_p(q', r_2, \ldots, r_{n+1}) - (q'-r_1)$. Using the result for $q', r_2, \ldots, r_{n+1}$ (which correspond to one of the cases that we already proved), we have $\mathcal{E}_p(r_1, \ldots, r_{n+1})  \ge \mathrm{min}\{q' + r_2 , q\} - (q'+r_1)$. If 
$\mathrm{min}\{q'+r_2, q\}=q'+r_2$, then we get $\mathcal{E}_p(r_1, \ldots, r_{n+1})\ge r_1 +r_2$ and we are done. Otherwise, we have $r_2 \ge q-q'$, and it follows that $u_2=\ldots u_{n+1}=p-1$. By Lemma (\ref{less_than_q}), we have
$$\mathcal{E}_p(r_1, r_2, \ldots, r_{n+1}) \ge \mathrm{min}\{\mathcal{E}_p(q', r_2, \ldots, r_n, (p-1)q'), (p-1)q'+\mathcal{E}_p(r_1, \ldots, r_n, v_{n+1})\}.$$ When $\mathrm{min}\{\mathcal{E}_p(q', r_2, \ldots, r_n, (p-1)q'), (p-1)q'+\mathcal{E}_p(r_1, \ldots, r_n, v_{n+1})\}=\mathcal{E}_p(q', r_2, \ldots, r_n, (p-1)q')$, we can apply the result to $q', r_2,\ldots, r_n, (p-1)q'$ (which correspond to one of the cases that we already proved) and we get $\mathcal{E}_p(r_1, \ldots, r_{n+1})\ge q$. When $\mathrm{min}\{\mathcal{E}_p(q', r_2, \ldots, r_n, (p-1)q'), (p-1)q'+\mathcal{E}_p(r_1, \ldots, r_n, v_{n+1})\}=(p-1)q'+ \mathcal{E}_p(r_1, \ldots, r_n, v_{n+1})$, note that this is $\ge (p-1)q'+r_n \ge 2(p-1)q' \ge pq'=q$.

Say that $I\ge 2$ and we may assume without loss of generality that $r_1, r_2 < q'$ and $r_{n+1} \ge q'$ ($r_{n+1}$ must be $\ge  q'$, since $r_n + r_{n+1} \ge q$). By Lemma (\ref{less_than_q}), we have 
$\mathcal{E}_p(r_1, \ldots, r_{n+1}) \ge \mathcal{E}_p(q', r_2, \ldots, r_n, u_{n+1}q')$, or $\mathcal{E}_p(r_1, \ldots, r_{n+1})\ge u_{n+1}q'+ \mathcal{E}_p(r_1, r_2, \ldots, r_n, v_{n+1})$. In the first case we can apply the inductive hypotheses (for $I-1$) to the $(n+1)$-tuple $q', r_2, \ldots, r_n, u_{n+1}q'$. In the second case, note that $\mathcal{E}_p(r_1, r_2, \ldots, r_n, v_{n+1}) \ge r_1$, therefore we obtain $\mathcal{E}_p(r_1, \ldots, r_{n+1})\ge u_{n+1}q'+ r_1 \ge r_1 + r_2$.

\end{proof}

{\bf The main body of the proof for the inequality $\ge $ in Theorem (\ref{main_theorem})}

First we prove the result for $p=2$. In this case, the only choice for $k_1, \ldots, k_{n+1}$ that satisfy condition (\ref{cond}) is $k_1 =\ldots = k_{n+1}=1$, and therefore we have $\mathcal{E}_p(k_1, \ldots, k_{n+1})=1$ and $\mathcal{E}_p(k_1+ \epsilon_1, \ldots, k_{n+1}+\epsilon_{n+1})=2$ for any choice of $(\epsilon_1, \ldots, \epsilon_{n+1}) \in \{0, 1\}^{n+1}$ that are not all equal to zero.Moreover, note that these are the smallest degree of any non-zero relations on $x^{k_1+\epsilon_1}, \ldots, x_n^{k_n + \epsilon_n}, f^{k_{n+1}}$, whether Koszul or not.
Thus, we have 
$$
{\mathcal Min}(k_1q+r_1, \ldots, k_{n+1}q+r_{n+1})=\mathrm{min}\{q + \sum_{i=1}^{n+1} r_i, 2q\}.
$$
According to Lemma (\ref{Lem}), a relation ${\mathcal A}=(a_1, \ldots, a_{n+1})^t$ on
\newline  $x_1^{k_1q+r_1}, \ldots, x_n^{k_nq+r_n}, f^{k_{n+1}q+r_{n+1}}$ restricts to a relation ${\mathcal A'}$ on $x_1^{k_1q}, \ldots, x_n^{k_nq}, f^{k_{n+1}q}$ of the form given by equation (\ref{sum}). If $(a_{1, M, l}, \ldots, a_{n+1, M, l})^t$ is a relation on $x_1^{k_1+\epsilon_1},\ldots, x_n^{k_n+\epsilon_n}, f^{k_{n+1}}$ with $\epsilon_1, \ldots, \epsilon_n \in \{0, 1\}$ not all equal to zero, then we have $\mathrm{deg}({\mathcal A}) \ge q \mathcal{E}_p(k_1+\epsilon_1, \ldots, k_n + \epsilon_n, k_{n+1})=2q$, and the desired conclusion holds.
Thus we may assume without loss of generality that the only terms in the sum (\ref{sum}) correspond to $M=x_1^{r_1}\cdots x_n^{r_n}$, and $(a_{1, M, l}, \ldots, a_{n+1, M, l})^t$ are non-Koszul relations on $x_1, \ldots, x_{n+1}, f$. Note that the only such relation is $(1, \ldots, 1, -1)^t$. The assumption that the relation $\mathcal{A'}$ is restricted from a relation on 
$x_1^{k_1q+r_1}, \ldots, x_n^{k_nq+r_n}, f^{k_{n+1}q+r_{n+1}}$
implies that $\sum_{M} g_{M, l}Ma_{n+1, M, l}^q$ is divisible by $f^{r_{n+1}}$. Since the sum consists of a single term, we can write $g:=g_{M, l}$, and we have $gx_1^{r_1}\cdots x_n^{r_n}\in (f^{r_{n+1}})$. This implies $g \in (f^{r_{n+1}})$, and therefore $\mathrm{deg}(\mathcal{A})= \mathrm{deg}(g) + \mathrm{deg}(x_1^{r_1} \cdots x_n^{r_n})+ q\, \mathcal{E}_p(k_1, \ldots, k_n, k_{n+1})\ge q + \sum_{i=1}^{n+1} r_i$.

From this point on we will assume $p\ge 3$.

Let $E:=\mathcal{E}_p(k_1, \ldots, k_{n+1})=\mathrm{min}\left( \lceil \frac{\sum_{i=1}^{n+1} k_i - n +1 }{2}\rceil, p\right)$. By the discussion following Lemma (\ref{char0}), we see that equation (\ref{ch0eq}) also applies to $k_1+\epsilon_1, \ldots, k_{n+1}+\epsilon_{n+1}$ for every $(\epsilon_1, \ldots, \epsilon_{n+1}) \in \{0, 1\}^{n+1}$, so we have
\begin{equation}\label{case1}
\mathcal{E}_p(k_1+ \epsilon_1, \ldots, k_{n+1}+\epsilon_{n+1})=\mathrm{min}(E+\lfloor \frac{I}{2}\rfloor, p)
\end{equation}
if $\sum_{i=1}^{n+1}k_i -n+1$ is odd, or 
\begin{equation}\label{case2}
\mathcal{E}_p(k_1+ \epsilon_1, \ldots, k_{n+1}+\epsilon_{n+1})=\mathrm{min}(E+\lceil \frac{I}{2}\rceil, p)
\end{equation}
if $\sum_{i=1}^{n+1}k_i-n+1$ is even, 
where $I=\sum_{i=1}^{n+1} \epsilon_i$.

Assume that we are in the case when equation (\ref{case1}) holds. The other case is similar. Then

 \begin{equation}\label{minimum}
{\mathcal Min}(k_1q+r_1, \ldots, k_{n+1}q+r_{n+1})=\mathrm{min}\{Eq + \lfloor \frac{I}{2} \rfloor q+ r_{i_1} + \ldots + r_{i_{n-I+1}} , pq\}
\end{equation}
where $I$ ranges through integers from $0$ to $n+1$ and $1 \le i_1 < \ldots < i_{n-I+1} \le n+1$  (note that the minimum will be achieved for an odd value of $I$).

Due to symmetry, we may assume with no loss of generality that $r_{n+1} \ge r_i$ for all $i \in \{1, \ldots, n\}$.

According to Lemma (\ref{Lem}), a relation ${\mathcal A}=(a_1, \ldots, a_{n+1})^t$ on
\newline  $x_1^{k_1q+r_1}, \ldots, x_n^{k_nq+r_n}, f^{k_{n+1}q+r_{n+1}}$ restricts to a relation ${\mathcal A'}$ on $x_1^{k_1q}, \ldots, x_n^{k_nq}, f^{k_{n+1}q}$ of the form given by equation (\ref{sum}).  Moreover, by Lemma (\ref{Cor}), we may assume that the relation $\mathcal{A'}$ is not Koszul. 
It follows that at least one of the terms of the sum in equation (\ref{sum}) is such that $(a_{1, M, l}, \ldots, a_{n+1, M, l})^t$ is  a non-Koszul relation on $x_1^{k_1+ \epsilon_1}, \ldots, x_n^{k_n + \epsilon_n}, f^{k_{n+1}}$,
 and therefore $\mathrm{deg}(a_{n+1, M, l})+k_{n+1}\ge \mathcal{E}_p(k_1+\epsilon_1, \ldots, k_n+\epsilon_n, k_{n+1})$ where $M=x_{i_1}^{r_{i_1}}\cdots x_{i_s}^{r_{i_s}}$ and $\epsilon_i=0 \Leftrightarrow i\in \{i_1, \ldots, i_s\}$.

 The assumption that ${\mathcal A'}$ is restricted from a relation on \newline 
$x_1^{k_1q+r_1}, \ldots, x_n^{k_nq+r_n}, f^{k_{n+1}q+r_{n+1}}$ means that $\sum_{M, l} g_{M, l} M a_{n+1, M, l}^q$ is divisible by $f^{r_{n+1}}$. We view $k[x_1, \ldots, x_n]$ as a free module over $k[x_1^q, \ldots, x_n^q]$ with basis consisting of the monomials $x_2^{i_2} \cdots x_n^{i_n} f^{i_{n+1}}$ with $0 \le i_2, \ldots, i_{n+1}\le q-1$. We see that every polynomial divisible by $f^{r_{n+1}}$ can be written 
in the form $\sum a_M M$ with $M$ ranging through monomials of the form $x_2^{i_2} \cdots x_n^{i_n}f^{i_{n+1}}$ as above, $a_M \in k[x_1^q, \ldots, x_n^q]$, and for each $M$ we have either $i_{n+1}\ge r_{n+1}$ or $a_M$ is divisible by $f^q$.
 We may assume without loss of generality that each entry $a_{n+1, M, l}$ of a non-Koszul relation from equation (\ref{sum}) is not divisible by $f$. Otherwise, the relation $(a_{1, M, l}, \ldots, a_{n+1, M, l})^t$ would be restricted from a relation on $x_1^{k_1 + \epsilon_1}, \ldots, x_n^{k_n + \epsilon_n}, f^{k_{n+1}+1}$, and we would have 
$\mathrm{deg}(a_{n+1, M, l}) + k_{n+1} \ge \mathcal{E}_p(k_1+ \epsilon_1, \ldots, k_n + \epsilon_n, k_{n+1}+1)$, which implies
$\mathrm{deg}(\mathcal{A}) =\mathrm{deg}(M)+ q(\mathrm{deg}(a_{n+1, M, l})+k_{n+1}) \ge {\mathcal Min}(d_1, \ldots, d_{n+1})$, as desired.

We will focus on the terms in the sum in equation (\ref{sum}) that correspond to non-Koszul relations and have minimal value for $I:=\sum_{i=1}^n \epsilon_i$ (and therefore minimal degree for $a_{n+1, M, l}$). 
Let $M_1, \ldots, M_N$ be all the monomials in (\ref{sum}) that are associated to these terms. Assume that $M_1=x_1^{r_1}x_{i_2}^{r_{i_2}}\cdots x_{i_s}^{r_{i_s}}$. Let $\{j_1, \ldots, j_t\}:=\{1, \ldots, n\}\, \backslash \, \{1, i_2, \ldots, i_s\}$. We may assume that the sum on the right hand side of equation (\ref{sum}) was written in such a way that it contains the fewest possible number of terms that involve the monomial $M_1$ and correspond to non-Koszul relations on $x_1^{k_1+\epsilon_1}, \ldots, x_n^{k_n + \epsilon_n}, f^{k_{n+1}}$ (where $\epsilon _i =0 \Leftrightarrow i \in \{1, i_2, \ldots, i_s\}$). 

Note that  $M_2, \ldots, M_N \in (x_{j_1}^{r_{j_1}}, \ldots, x_{j_t}^{r_{j_t}})$ (since they involve the same number of variables as $M_1$, they must involve at least one variable not in $M_1$).  Say that the coefficients of the terms corresponding to $M_1$ in the sum (\ref{sum}) are $g_{M_1, 1}, \ldots, g_{M_1, \lambda}$.  Fix a monomial $\nu$ from the expansion of $g_{M_1, 1}M_1$ using $x_2, \ldots, x_n, f$ as variables.
We claim that either $\nu \in (x_{j_1}^{r_{j_1}}, \ldots, x_{j_y}^{r_{j_t}}, x_1^q, x_{i_2}^q, \ldots, x_{i_s}^q, f^{r_{n+1}})$, or else the conclusion of the theorem holds.

Assume $\nu \notin (x_1^q, \ldots, x_n^q)$.  Then it is part of the basis for $k[x_1, \ldots, x_n]$ as a free $k[x_1^q, \ldots, x_n^q]$-module. According to the discussion above, one of the following situations must occur: $\nu$ is divisible by $f^{r_{n+1}}$, or $a_{n+1, M_1, 1}$ is divisible by $f$, or $\nu a_{n+1, M_1, 1}^q$ must cancel out in the sum $\sum_{M, l} g_{M, l} M a_{n+1, M, l}^q$ . Since the claim holds in the first situation, and the second situation has been eliminated in the discussion above, we may assume that the third possiblity occurs.

Now we consider the other terms of the sum $\sum_{M, l} g_{M, l} M a_{n+1, M, l}^q$ such that the monomial $\nu $ appears in the expansion of $g_{M, l}M$ after possibly dividing out $q$th powers (so we are looking for monomials in the expansion of a $g_{M, l} M$ that are of the form ($q$th power) $\cdot \nu$). First assume that $\nu $ appears in a term that corresponds to a non-Koszul relation $(a_{1, M}, \ldots, a_{n+1}, M)^t$ and $M \ne M_1$.
If  $M\notin \{M_1, \ldots, M_N\}$ (so that $M$ has $\sum_{i=1}^n \epsilon'_i>I$), then we claim that the conclusion of the theorem follows.  Indeed, in this case we have $\mathrm{deg}(g_{M_1, 1}M_1)\le \mathrm{deg}(g_{M, l} M)$, and therefore $\mathrm{deg}(a_{n+1, M_1, 1})\ge \mathrm{deg}(a_{n+1, M, l})$. Since we have $k_{n+1}+\mathrm{deg}(a_{n+1, M, l})=\mathcal{E}_p(k_1+\epsilon'_1, \ldots, k_n+\epsilon'_n, k_{n+1})\ge (\mathrm{min}\{E+\lfloor \frac{I+1}{2}\rfloor , p\})q$, it follows that $\mathrm{deg}(\mathcal{A})\ge \mathrm{deg}(M_1) + (\mathrm{min}\{E+\lfloor \frac{I+1}{2}\rfloor , p\})q\ge {\mathcal Min}(d_1, \ldots, d_{n+1})$, which is the desired conlcusion.
 If $M\in \{M_2, \ldots, M_N\}$, then the term from the expansion of $g_M M$ must be $\nu$ rather than a ($q$th power)$\cdot \nu$ (because in this case $g_M M $ and $g_{M_1} M_1$ must have the same degree), and since $M\in (x_{j_1}^{r_{j_1}}, \ldots, x_{j_t}^{r_{j_t}})$ it follows that $\nu \in (x_{j_1}^{r_{j_1}}, \ldots, x_{j_t}^{r_{j_t}})$, and the claim holds.

  Now assume that none of the previous possibilities holds, and therefore the term $\nu a_{n+1, M_1, 1}$ must cancel with terms in the sum $\sum_{M, l} g_{M, l} M a_{n+1, M, l}^q$ that have $M=M_1$ or correspond to Koszul relations $(a_{1, M, l}, \ldots, a_{n+1, M, l})^t$.

 Say that $g_{M_1, 1}M_1, \ldots, g_{M_1, s}M_1$ are all the terms corresponding to non-Koszul relations where the monomial $\nu$ occurs with coefficients $\alpha^q_1, \ldots, \alpha^q_s \in k$ (using the assumption that $k$ is a perfect field).  Moreover, let $N_1, \ldots, N_u$ be the monomials corresponding to Koszul relations such that  (a $q$th power) $\cdot \nu $ is a term in the expansion of a $g_N N$, and let $\beta_1^q, \ldots, \beta _s^q$ be the corresponding coefficients. 
 The assumption that the terms containing this monomial cancel
in the sum (\ref{sum}) implies that  $\alpha_1 a_{n+1, M_1, 1}=\alpha_2 a_{n+1, M_1, 2}+ \ldots + \alpha_s a_{n+1,M_1, s}+v$, where 
$v=\sum_{m=1}^u \beta _m v_ma_{n+1, N_m}$ is the term coming from the Koszul relations. Specifically, $v_m$ is such that $g_{N_m}N_m$ has a term equal to $v_m^q \nu $ in its expansion, and $a_{n+1, N_m}\in (x_1^{k_1+\epsilon_{1, m}}, \ldots, x_n^{k_n+\epsilon_{n, m}})$ is the last entry of a relation  restricted from a Koszul relation on $x_1^{k_1+\epsilon_{1, m}}, \ldots, x_n^{k_n + \epsilon_{n, m}}, f^{k_{n+1}}$ where $\epsilon_{i, m}=0$ if and only if $x_i^{r_i}$ is a factor in $N_m$.
Note that $v_m (a_{1, N_m}, \ldots, a_{n+1, N_m})^t$ is a relation on o$x_1^{k_1}, \ldots, x_n^{k_n}, f^{k_{n+1}}$ which is restricted from a relation on $x_1^{k_1+\epsilon_1}, \ldots, x_n^{k_n + \epsilon_n}, f^{k_{n+1}}$ (recall that
 $\epsilon_i =0 \Leftrightarrow i \in\{1, i_2, \ldots, i_s\}$). 
To see this, say that $ N_m=x_{l_1}^{r_{l_1}}\cdots x_{l_v}^{r_{l_v}}$.
For every $j \in \{1, \ldots, n\} \, \backslash \, \{1, i_2, \ldots, i_s\}$, if $j \notin \{l_1, \ldots, l_v\}$  we have $\epsilon_{j, m}=1$, and if $j\in \{l_1, \ldots, l_v\}$ then we must have $v_m$ divisible by $x_j$ (since $N_m g_{N_m}$ is divisible by $x_j^{r_j}$ but $\nu$ is not).

 It follows that the relation 
$(a_{1, M_1, 1}, \ldots, a_{n+1, M_1, 1})^t$ from equation (\ref{sum}) can be replaced by 
$$
\sum_{j=2}^s \frac{\alpha_j}{\alpha _1} \left( \begin{array}{c} a_{1, M_1, j} \\ \vdots \\ a_{n+1, M_1, j}\\ \end{array}\right) + K
$$
where $K$ is a relation on $x_1^{k_1}, \ldots, x_n^{k_n}, f^{k_{n+1}}$ which is restricted from a Koszul relation on $x_1^{k_1+ \epsilon_1}, \ldots, x_n^{k_n + \epsilon_n}, f^{k_{n+1}}$. This contradicts the assumption that the sum in equation (\ref{sum}) was written so that it contains the fewest possible number of terms corresponding to non-Koszul relations that involve the monomial $M_1$. This concludes the proof of the claim.

Consider the case when $\{i_2, \ldots, i_s\} \ne \emptyset$ (in other words the monomial $M_1$ consists of more than a single variable).

Rename $g:=g_{M_1, 1}$. Since $x_2, \ldots, x_n, f$  is a regular sequence, $$gM_1 \in (x_{j_1}^{r_{j_1}}, \ldots, x_{j_t}^{r_{j_t}}, x_1^q, x_{i_2}^q, \ldots, x_{i_s}^q, f^{r_{n+1}})$$ implies that  
\newline 
$gx_1^{r_1} \in (x_{j_1}^{r_{j_1}}, \ldots, x_{j_t}^{r_{j_t}}, x_{i_2}^{q-r_{i_2}}, \ldots, x_{i_s}^{q-r_{i_s}}, f^{r_{n+1}})$ (note that $x_1^q$ is redundant among the generators of the ideal on the right hand side and can be omitted). We view this as a relation on $x_1^{s_1}, \ldots, x_n^{s_n}, f^{s_{n+1}}$ where $s_i= r_i $ for $i \in \{1, j_1, \ldots, j_t, n+1\}$ and $s_i =q-r_i$ for $i \in \{ i_2, \ldots, i_s\}$.
This implies either $g \in (x_{j_1}^{r_{j_1}}, \ldots, x_{j_t}^{r_{j_t}}, x_{i_2}^{q-r_{i_2}}, \ldots, x_{i_s}^{q-r_{i_s}}, f^{r_{n+1}})$ if the above is a Koszul relation, or
$\mathrm{deg}( g)+ r_1 \ge \mathcal{E}_p(s_1, \ldots, s_n, s_{n+1})$ otherwise.  In the first case, we have $\mathrm{deg}(g) \ge r_l$ for some $l \in \{j_1, \ldots, j_t, n+1\}$ or $\mathrm{deg}(g) \ge q-r_l$ for some $l \in \{i_2, \ldots, i_s\}$. In the second case we have $\mathrm{deg}(g) + r_1 \ge \mathrm{min}(s_i + s_j, q)$.

We have $\mathrm{deg}(\mathcal{A}) = \mathrm{deg}(M_1)+ \mathrm{deg}(g) + q(\mathrm{deg}(a_{M_1, n+1, 1}+k_{n+1})$, and $\mathrm{deg}(a_{M_1, n+1, 1}+k_{n+1})=\mathcal{E}_p(k_1+ \epsilon_1, \ldots, k_n+\epsilon_n, k_{n+1})=\mathrm{min}\{E+ \lfloor \frac{I}{2}\rfloor, p\}$, where $\epsilon_i=1$ for $i \in \{j_1, \ldots, j_t\}$ and $I=t=n-s$. Thus $\mathrm{deg}(\mathcal{A})=r_1 + r_{i_2} + \ldots + r_{i_s} + \mathrm{deg}(g) + q \mathrm{min}\{E+\lfloor \frac{I}{2}\rfloor, p\} $.
We distinguish six possible cases based on the discussion above. We refer to equation (\ref{minimum}) for explanation of the last inequality in each of the cases below.

{\bf 1.} $\mathrm{deg}(g) \ge r_j$ for some $j \in \{j_1, \ldots, j_t, n+1\}$. Then we have $\mathrm{deg}(\mathcal{A})\ge r_1 + r_{i_2} + \ldots + r_{i_s} + r_j +\mathrm{deg}(g) +q\,  \mathrm{min}\{E+\lfloor \frac{I}{2}\rfloor, p\} \ge {\mathcal Min}(d_1, \ldots, d_{n+1})$ since the number of $r_i$'s in the sum is $s+1= n-I+1$.

{\bf 2.} $\mathrm{deg}(g) \ge q-r_i$ for some $i \in \{i_2, \ldots, i_s\}$. We may assume without loss of generality that $i=i_s$. Then we have 
$\mathrm{deg}(\mathcal{A})\ge r_1 + r_{i_2} +  \ldots + r_{i_{s-1}} + q + q\,  \mathrm{min}\{E+ \lfloor \frac{I}{2}\rfloor , p\}  \ge  r_1 + r_{i_2} +  \ldots + r_{i_{s-1}}+q\, \mathrm{min}\{E+\lfloor \frac{I+2}{2}\rfloor , p\} \ge {\mathcal Min}(d_1, \ldots, d_{n+1})$, since the number of $r_i$'s in the sum is $s-1=n-(I+2)+1$.

{\bf 3.} $\mathrm{deg}(g)+r_1 \ge r_j + r_{j'}$ for some $j, j' \in \{j_1, \ldots, j_t, n+1\}$. Then we have $\mathrm{deg}(\mathcal{A})\ge  r_{i_2} + \ldots + r_{i_s} + r_j + r_{j'} + q\,  \mathrm{min}\{E+\lfloor \frac{I}{2} \rfloor , p\} \ge {\mathcal Min}(d_1, \ldots, d_{n+1})$ since the number of $r_i$'s in the sum is $s+1 = n-I +1$.

{\bf 4.} $\mathrm{deg}(g)+r_1 \ge r_j + (q-r_i)$ for some $j \in \{j_1, \ldots, j_t, 1, n+1\}$, $i\in \{i_2, \ldots, i_s\}$. Assume $i=i_s$. Then we have 
$\mathrm{deg}(\mathcal{A})\ge r_{i_2} + \ldots + r_{i_{s-1}} +r_j + q + q \mathrm{min}\{E+\lfloor \frac{I}{2}\rfloor, p\}  \ge  r_{i_2} + \ldots + r_{i_{s-1}} +r_j+q \mathrm{min}\{E+\lfloor \frac{I+2}{2}\rfloor, p\}  \ge {\mathcal Min}(d_1, \ldots, d_{n+1})$ because the number of $r_i$'s in the sum is $s-1=n-(I+2)+1$.

{\bf 5.} $\mathrm{deg}(g)+r_1 \ge (q-r_i)+(q-r_{i'})$ for some $i, i'\in \{i_2, \ldots, i_s\}$. Assume $\{i, i'\}=\{s-1, s\}$. Then we have 
$\mathrm{deg}(\mathcal{A})\ge r_{i_2} + \ldots + r_{i_{s-2}} + 2q + q \, \mathrm{min}\{E+\lfloor \frac{I}{2} \rfloor, p\}  \ge  r_{i_2} + \ldots + r_{i_{s-2}}+q \mathrm{min}\{E+\lfloor \frac{I+4}{2}\rfloor, p\} \ge {\mathcal Min}(d_1, \ldots, d_{n+1})$, because the number of $r_i$'s in the sum is $s-3=n-(I+4)+1$.

{\bf 6.} $\mathrm{deg}(g)+r_1 \ge q$. Then we have $\mathrm{deg}(\mathcal{A})\ge r_{i_2} + \ldots + r_{i_{s}}+ q+q \, \mathrm{min}\{E+\lfloor \frac{I}{2} \rfloor, p\}  \ge r_{i_2} + \ldots + r_{i_s} +q \, \mathrm{min}\{E+\lfloor \frac{I+2}{2} \rfloor, p\}\ge {\mathcal Min}(d_1, \ldots, d_{n+1})$, because the number of $r_i$'s in the sum is $s-1= n-(I+2)+1$.

Now consider the case when $M_1=x_1^{r_1}$. Then we have $g x_1^{r_1} \in (x_2^{r_2}, \ldots, x_n^{r_n}, f^{r_{n+1}})$. If 
$g \in (x_2^{r_2}, \ldots, x_n^{r_n}, f^{r_{n+1}})$, then we proceed as in case {\bf 1.} above. Otherwise, we have
$\mathrm{deg}(M_1)+r_1 \ge {\mathcal E}(r_1, r_2, \ldots, r_{n+1})$. We may assume without loss of generality that $\mathrm{max}\{r_i + r_j\}= r_n + r_{n+1}$.
If $r_n + r_{n+1} \ge q$, then we can use Lemma (\ref{remainders}) and proceed as in case {\bf 3.} or case {\bf 6.} above. Assume that $r_n + r_{n+1}  < q$ (so that Lemma (\ref{remainders}) cannot be applied).
Note that we have $$\mathrm{deg}({\mathcal A}) = \mathrm{deg}(g)+r_1 + \mathcal{E}_p(k_1, k_2+1, \ldots, k_n+1, k_{n+1}) \ge$$
$$ \mathcal{E}_p(r_1, \ldots, r_{n+1})+Eq +q \, \mathrm{min}\{\lfloor \frac{n-1}{2} \rfloor, p\} . $$ Choose $u=q-(r_n+r_{n+1}) >0$ and note that we have 
$\mathcal{E}_p(r_1, r_2, \ldots, r_{n+1}) \ge \mathcal{E}_p(r_1 , r_2, \ldots, r_{n+1}+u)-u$ from Lemma (\ref{inequalities}). Note that $r_{n+1} + u \le q$, $r_n+ r_{n+1}+u = q$, and we can apply Lemma (\ref{remainders}) to the $(n+1)$-tuple $(r_1, r_2, \ldots, r_{n+1}+u)$. We have
$$\mathcal{E}_p(r_1, \ldots, r_{n+1}) \ge \mathrm{min}\{ r_i + r_j\}-q+r_n+r_{n+1}.$$ It follows that 
$$\mathrm{deg}({\mathcal A}) \ge q \, \mathrm{min}\{E+ \lfloor \frac{n-1}{2} \rfloor,  p\}  -q + r_i + r_j + r_n + r_{n+1}\ge
$$
$$
q\, \mathrm{min}\{E + \lfloor \frac{n-3}{2}\rfloor, p\}  + r_i + r_j + r_ n + r_{n+1}.$$
We see that $\mathrm{deg}(\mathcal{A}) \ge {\mathcal Min}(d_1, \ldots, d_{n+1})$ by choosing $I=n-3$ in equation (\ref{minimum}).

\section{Diagonal F-thresholds}
Let $(A, m)$ denote a standard graded Artinian local ring with maximal homogeneous ideal $m$. The top socle degree of $A$ is the largest degree of a nonzero element of $A$. 

We observe that knowing $\mathcal{E}_p(d_1, \ldots , d_{n+1})$ for every $d_1, \ldots, d_{n+1}$ will allow us to find the top socle degree of any ring of the form
$$
R=\frac{k[x_1, \ldots,x_{n+1}]}{(x_1^{K_1}, \ldots, x_{n+1}^{K_{n+1}}, x_1^a + \ldots + x_{n+1}^a)}
$$
(note that $x_{n+1}$ denotes a variable; we are {\bf no longer} using the convention $x_{n+1} = x_1 + \ldots x_n$ that was in effect in the earlier sections).
Let the top socle degree of this ring be denoted tsd$(K_1, \ldots, K_{n+1}; a)$.
Of particular interest is the case $K_1 = \ldots = K_{n+1}=q=p^e$, in which case we will be finding the top socle degrees of Frobenius powers of the maximal ideal in a diagonal hypersurface ring.

\begin{theorem}\label{socle_degrees}
Let $K_1, \ldots, K_{n+1}, a >0$ be integers. Write $K_i = ad_i + e_i$ with $0 \le e_i \le a-1$.
With notation as above, we have
$$
\mathrm{tsd}(K_1, \ldots, K_{n+1}; a)=
$$
$$
\mathrm{max}\left\lbrace a\left( \sum_{i=1}^{n+1} (d_i + \epsilon _i -1) -\mathcal{E}_p(d_1 + \epsilon _1, \ldots, d_{n+1}+\epsilon _{n+1})+1\right) 
  + \sum_{i=1}^{n+1} g_i\right\rbrace 
$$
where the maximum is taken over all choices of $(\epsilon_1, \ldots, \epsilon_{n+1}) \in \{0, 1\}^{n+1}$, $g_i=a-1$ when $\epsilon _i =0$ and $g_i = e_i -1$ when $\epsilon _i =1$.
\end{theorem}
\begin{proof}
For every monomial $x_1^{j_1} \cdots x_{n+1}^{j_{n+1}}$, write $j_i =af_i + g_i$ with $0 \le g_i \le a-1$. Define a ${\bf Z}_a \times {\bf Z}_a \times \cdots \times {\bf Z}_a$ grading on $k[x_1, \ldots, x_{n+1}]$ by letting $\mathrm{deg}(x_1^{j_1} \cdots x_{n+1}^{j_{n+1}})=(g_1, \ldots, g_{n+1})$. We claim that
$$
x_1^{j_1} \cdots x_{n+1}^{j_{n+1}}\in (x_1^{K_1}, \ldots, x_{n+1}^{K_{n+1}}, x_1^a + \ldots + x_{n+1}^a) \Leftrightarrow 
$$
$$
x_1^{f_1}\cdots x_{n+1}^{f_{n+1}} \in (x_1^{d_1 + \epsilon_1},  \ldots,  x_{n+1}^{d_{n+1} + \epsilon _{n+1}}, x_1 + \ldots +x_{n+1})
$$
where $\epsilon _j =0$ if $e_j \le g_j$ and $\epsilon _j=1$ if $e_j >g_j$.
Indeed assume that 
\begin{equation}\label{socle}
x_1^{af_1+g_1}\cdots x_{n+1} ^{af_{n+1} + g_{n+1}}=F_1x_1^{ad_1+e_1} + \ldots + F_{n+1} x_{n+1}^{ad_{n+1}+e_{n+1}}+ F_{n+2}(x_1^a + \ldots x_{n+1}^a)
\end{equation}
with $F_1, \ldots, F_{n+2}\in k[x_1, \ldots, x_{n+1}]$. We can assume that each term in the sum on the right hand side of the equation is homogeneous of degree $(g_1, \ldots, g_{n+1})$ in the ${\bf Z}_a \times {\bf Z}_a \times \cdots \times {\bf Z}_a$ grading, which means that it has the form $x_1^{g_1} \cdots x_{n+1}^{g_{n+1}}\cdot (\mathrm{polynomial \ in \ } x_1^a, \ldots, x_{n+1}^a)$. Fix a $i\le n+1$. If $e_i\le g_i$, this means that $F_i$ is of the form $x_1^{g_1} \cdots x_i^{g_i-e_i}\cdots x_{n+1}^{g_{n+1}} \cdot (\mathrm{polynomial\ in\ } x_1^a, \ldots, x_{n+1}^a)$. If $e_i > g_i$, then $F_i$ is of the form $x_1^{g_1}\cdots x_i^{a -e_i+g_i}\cdots x_{n+1}^{g_{n+1}} \cdot (\mathrm{polynomial\ in}\  x_1^a, \ldots, x_{n+1}^a)$. Now we can simplify $x_1^{g_1}\cdots x_{n+1}^{g_{n+1}}$ on both sides of equation (\ref{socle}) and let $y_i:=x_i^a$. The conclusion follows by viewing the two sides of the resulting equation as polynomials in the variables $y_1, \ldots, y_{n+1}$. 

Therefore we have 
$$
\mathrm{tsd}(K_1, \ldots, K_{n+1}; a) = \mathrm{sup} \left\lbrace a\, \mathrm{tsd}(d_1+\epsilon_1, \ldots, d_{n+1}+\epsilon_{n+1}; 1)+\sum_{i=1}^{n+1} g_i \right\rbrace
$$
where the supremum is over all the choices of $(\epsilon_1, \ldots, \epsilon_{n+1})\in \{0, 1\}^{n+1}$, and $g_i=a-1$ when $\epsilon_i=0$, $g_i=e_i-1$ when $\epsilon_i=1$ (one can use these choices for $g_i$ due to the fact that $\mathrm{tsd}$ is an increasing function in $K_1, \ldots, K_{n+1}$).

Now note that $\mathrm{tsd}(d_1, \ldots, d_{n+1}; 1)$ is the largest degree $j$ such that the map
$$
\times L : \left(\frac{k[x_1, \ldots, x_{n+1}]}{(x_1^{d_1}, \ldots, x_{n+1}^{d_{n+1}})}\right)_{j-1} \rightarrow \left(\frac{k[x_1, \ldots, x_{n+1}]}{(x_1^{d_1}, \ldots, x_{n+1}^{d_{n+1}})}\right)_j
$$
is not surjective, where $L=x_1 + \ldots + x_{n+1}$. By the perfect pairing $A_j \times A_{S-j} \rightarrow A_{S}$ in the graded Artinian Gorenstein ring $A=k[x_1, \ldots, x_{n+1}]/(x_1^{d_1}, \ldots, x_{n+1}^{d_{n+1}})$, we see that $j=S-i$, where $S=\sum_{i=1}^{n+1})d_i-1)$ is the socle degree of $A$, and $i$ is the smallest degree such that the map $\times L: 
A_i \rightarrow A_{i+1}$ is not injective. But it is not hard to see that $i+1=\mathcal{E}_p(d_1, \ldots, d_{n+1})$. It follows that $\mathrm{tsd}(d_1, \ldots, d_{n+1}; 1)=\sum_{i=1}^{n+1}(d_i-1)-\mathcal{E}_p(d_1, \ldots, d_{n+1})+1$.
\end{proof}

F-threshods of ideals with respect to other ideals were introduced in \cite{MTW}, where it is shown that they give an analogue of the jumping coefficients of multiplier ideals in characteristic zero. 
We remind the reader that 
given two ideals $\frak{a}, J \subset \frak{m}$ in a local ring $(R, \frak{m})$ of characteristic $p>0$, such that $\frak{a}\subseteq \sqrt{J}$, the F-threshold of $\frak{a}$ with respect to $J$ is defined as
$$
c^J(\frak{a}):=\lim_{e\rightarrow \infty} \frac{\nu _{\frak{a}}^J(p^e)}{p^e}, \ \mathrm{where}
 \ \ \ \nu _{\frak{a}}^J(p^e):=\mathrm{max}\{ r \, | \, \frak{a}^r \not\subset J^{[p^e]}\}.
$$
When $\frak{a} = J =\frak{m}$, $c^{\frak{m}}(\frak{m})$ is called the diagonal F-threshold of $R$. It is observed in \cite{Jinjia} that when $\frak{a}=\frak{m}$, and $R$ is a standard graded ring, $\nu_{\frak{m}}^J(q)$ is the top socle degree of $R/J^{[q]}$. Therefore we can apply the result of Theorem \ref{socle_degrees} to calculate the diagonal F-threshold of diagonal hypersurface rings. 
\begin{theorem}
Let
$\displaystyle R=\frac{k[x_1, \ldots, x_{n+1}]}{(x_1^a+\ldots + x_{n+1}^a)}$ where $k$ is a field of positive characteristic $p$ and $a$ is a positive integer not divisible by $p$. 
Then
$$
c^{\frak{m}}(\frak{m})=n+1-a \mathcal{M}
$$
where $\mathcal{M}$ is equal to
$$\mathrm{min}\{\lceil \frac{(n+1)\kappa-n+1}{2}\rceil \cdot \frac{1}{p^e}+\frac{(n+1)s}{ap^e}, \lceil \frac{(n+1)\kappa-n+2}{2} \rceil \cdot\frac{1}{p^e}+\frac{ns}{ap^e}, 
$$
$$\lceil \frac{(n+1)\kappa+1}{2}\rceil \cdot \frac{1}{p^e}+\frac{s}{ap^e}, \lceil \frac{(n+1)\kappa+2}{2} \rceil \cdot \frac{1}{p^e}, \frac{1}{p^{e-1}}\}
$$
where $e$ is the smallest exponent such that $p^e \ge a$, $\kappa=\lfloor \frac{p^e}{a}\rfloor$, and $s=p^e-\kappa a$ is the remainder of $p^e$ modulo $a$.
\end{theorem}
\begin{proof}
We have $\displaystyle c^{\frak{m}}(\frak{m})= \lim_{q=p^e \rightarrow \infty} \frac{tsd(q, \ldots, q; a)}{q}$. Since we know from Proposition 2.3 in \cite{Jinjia}, that the limit exists, it suffices to consider the subsequence consisting of $q\equiv 1$ (mod $a$). 
Let $q=d(q)a +1$, with $d(q)=(q-1)/a$.

According to Theorem (\ref{socle_degrees}),
$$
\mathrm{tsd}(q, \ldots, q; a)=
$$
$$
\mathrm{max}\{ a\left[ \sum _{i=1}^{n+1} (d(q)+\epsilon_i -1) - \mathcal{E}_p(d(q)+ \epsilon_1, \ldots, d(q)+\epsilon_{n+1})+1\right] 
-(a-1)(n+1-I)\}=
$$
$$
\mathrm{max} \{ a(n+1)(d(q)-1)-a\mathcal{E}_p(d(q)+\epsilon_1, \ldots, d(q)+\epsilon_{n+1})+a+a(n+1)-(n+1)+I\}
$$
where the maximum is taken over all the choices of $\epsilon_1, \ldots, \epsilon_{n+1}\in \{0, 1\}$, and $I=\sum_{i=1}^{n+1} \epsilon_i$.
Upon dividing by $q$ and taking the limit, we see that the last four terms above have zero contribution in the limit. Thus, we may take $\epsilon_1= \ldots = \epsilon_{n+1} =0$  in order to achieve the maximum of the remaining terms. We have
$$
\mathrm{ft}(R)=a(n+1) \lim_{q \rightarrow \infty}\frac{d(q)}{q} - a\lim_{q\rightarrow \infty} \frac{\mathcal{E}_p(d(q), \ldots, d(q))}{q}
=n+1- a \lim_{q\rightarrow \infty} \frac{\mathcal{E}_p(d(q), \ldots, d(q))}{q}.
$$
Let $e$ be the smallest exponent such that $p^e \ge a$. Let $s$ denote the remainder of $p^e$ modulo $a$ ($0 \le s \le a-1$), and let $\displaystyle q'=\frac{q}{p^e}$. We can write
 $d(q)=k(q) q' + r(q)$ with $\displaystyle k(q)=\frac{p^e-s}{a}$ and $\displaystyle r(q)=\frac{sq'-1}{a}$. Note that $\displaystyle k(q)=\lfloor \frac{p^e}{a} \rfloor=\kappa$. Also note that $1 \le \kappa \le p-1$ (since $p^{e-1} < a\le p^e$) and $0 \le r(q)<q'$ because $s<a$. We can therefore apply Theorem (\ref{main_theorem}) and
we have
$$
\mathcal{E}_p(d(q), \ldots, d(q))=\mathrm{min}\{ \mathcal{E}_p(\kappa+\epsilon_1, \ldots, \kappa+\epsilon_{n+1})q' + (n+1-I)r(q)\}
$$
$$
=\mathrm{min}\{\lceil \frac{(n+1)\kappa +I -n+1}{2}\rceil q' + (n+1-I)r(q), pq'\}
$$
where the the first minimum is over all choices of $\epsilon_i \in \{0, 1\}$, and $I=\sum_{i=1}^{n+1} \epsilon_i$ and the second minimum is over all the choices of $0 \le I \le n+1$.

Let $\displaystyle m(I):= \lceil \frac{(n+1)\kappa +I -n+1}{2}\rceil q' + (n+1-I)r(q)$.
Note that $\displaystyle m(I+2)-m(I)=q'-2r(q)=\frac{aq'-2sq'-2}{a}$. It follows that $m(I+2)-m(I)>0$ for all $0 \le I \le n-1$  if $a>2s$ and $q$ is  sufficiently large, and $m(I+2)-(I)<0$ for all $0 \le I \le n-1$ if $a\le 2s$. Therefore the minimum of $m(I)$ when $0 \le I \le n+1$ is achieved for $I \in \{0, 1, n, n+1\}$.

The conclusion now follows by observing that $\displaystyle \lim_{q\rightarrow \infty} \frac{r(q)}{q}=\frac{s}{ap^e}$.
\end{proof}

\section{The weak Lefschetz property}

Recall that a standard graded algebra $A=\oplus A_i$ is said to have the weak Lefschtez property (WLP) if there exists a $L\in A_1$ such that the map $\times L : A_i \rightarrow A_{i+1}$ has maximal rank (i.e. is either injective or surjective) for every $i$.

Let $k$ be an infinite field of positive characterisitc $p$, and let $d_i = k_i q +r_i$ satisfy the assumptions in Theorem (\ref{main_theorem}), with $q=p^e$ a power of $p$. We study the following question:

\begin{question}\label{wlp_question}
For what values of $d_1,\ldots, d_{n+1}$ as above does the ring
 $\displaystyle A=\frac{k[x_1, \ldots, x_{n+1}]}{(x_1^{d_1}, \ldots, x_{n+1}^{d_{n+1}})}$ does have the weak Lefschetz property (WLP)?
\end{question}

If the field $k$ has characteristic zero, then it is known (see \cite{St}, \cite{W}) that all the monomial complete intersections rings $A$ have WLP, but the story is much different in positive characteristic. 
The question (\ref{wlp_question}) has been investigated in \cite{BK}) for $n=2$ when $d_1=d_2=d_3$ and in \cite{LZ} for $n=2$ in the general case. The question was also answered in \cite{KV} for the case $n\ge 3$ when $d_1 = \ldots d_{n+1}$. The results of this section generalize those in \cite{KV}. Closely related problems are studied in \cite{CN} and \cite{C2}. A survey of the history and recent developments related to the weak Lefschetz property is given in \cite{MN}.

We see that when $n \ge 5$ there are no values of $d_1, \ldots, d_{n+1}$ that satisy the assumption in Theorem (\ref{main_theorem}) with $q>1$ such that $A$ has WLP (see Corollary (\ref{corWLP}); when $n=4$ the only values are for $p=3$, $(d_1, \ldots, d_5)=(4, 4, 4, 4, 5)$ or a permuation of this (see Proposition (\ref{propWLP}). 
 When $n=3$ we give a complete characterization of those values of $d_1, \ldots, d_{n+1}$ that satisfy the assumption of Theorem (\ref{main_theorem}) and such that the ring $A$ has the weak Lefschetz property (see Proposition (\ref{prop2WLP}). These results generalize our previous work in \cite{KV}, where we considered the case $d_1 = \ldots = d_{n+1}$. Closely related problems are studied in \cite{CN} and and \cite{C2}.

Recall (see for example Corollary 2.2 in \cite{KV}) that $A$ has the weak Lefschetz property if and only if 
\begin{equation}\label{wlp}
\mathcal{E}_p(d_1, \ldots, d_{n+1}) \ge \lceil \frac{\sum_{i=1}^{n+1} d_i - n +1}{2} \rceil
\end{equation}

Let 
$$
E:=\lceil \frac{\sum_{i=1}^{n+1} k_i - n+1 }{2}\rceil
$$
Then for every $0 \le I \le n+1$  we have $\mathcal{E}_p(d_1, \ldots, d_{n+1}) \le \lceil \frac{\sum_{i=1}^{n+1} k_i + I -n+1}{2} \rceil q + r_1 + \ldots + r_{n-I+1}$.
We will discuss two cases depending on the parity of $\sum_{i=1}^{n+1} k_i - n +1$. For convenience of notation we will assume that $r_1 \le r_2 \le \ldots \le r_{n+1}$.

{\bf Case 1:} Assume that $\sum _{i=1}^{n+1} k_i -n +1$ is odd. Then we have $\sum_{i=1}^{n+1} k_i = 2E +n -2$, and equation (\ref{wlp}) implies that if the weak Lefschetz property holds, then 
$$
\lceil \frac{\sum_{i=1}^{n+1} d_ i - n+1}{2} \rceil = \lceil \frac{(2E+n-2)q+ \sum_{i=1}^{n+1} r_i - n+1} {2} \rceil \le (E+ \lfloor \frac{I}{2} \rfloor ) q + r_1 + \ldots + r_{n+1-I},
$$
or equivalently
\begin{equation}\label{genI}
(n-2-2\lfloor \frac{I}{2} \rfloor) q + r_{n+2-I}+ \ldots + r_{n+1} \le r_1 + \ldots + r_{n+1-I} + n-1
\end{equation}
for every $I =0, \ldots, n+1$.

{\bf Case 2:} Assume that $\sum_{i=1}^{n+1} k_i -n +1$ is even. Then we have $\sum_{i=1}^{n+1} k_i = 2E+n-1$, and equation (\ref{wlp}) implies that if the weak Lefschetz property holds, then 
$$
\lceil \frac{\sum_{i=1}^{n+1} d_i -n +1}{2} \rceil = \lceil \frac{(2E+n-1)q+ \sum_{i=1}^{n+1} r_i - n+1}{2} \rceil \le (E+\lceil \frac{I}{2}\rceil)q + r_1 + \ldots + r_{n+1-I},
$$
or equivalently
\begin{equation}\label{genIeven}
(n-1-2\lceil \frac{I}{2} \rceil)q + r_{n+2-I} + \ldots + r_{n+1} \le r_1 + \ldots + r_{n+1-I} + n-1
\end{equation}
for every $I =0, \ldots, n+1$.

Furthermore note that $\mathcal{E}_p(d_1, \ldots, d_{n+1})\le \mathcal{E}_p(pq, \ldots, pq)=pq$, and thus if the weak Lefschetz property holds then we must have \begin{equation}\label{pq}
pq \ge \lceil \frac{\sum_{i=1}^{n+1} d_i - n +1 }{2} \rceil
\end{equation}

We will show that when $n$ is sufficiently large, equation (\ref{wlp}) implies that $q=1$ or $q=2$. For smaller values of $n$ (but still $n\ge 4$) we will have to consider a few more possible values of $q$.

\begin{lemma}\label{largen}
Assume that $A$ has the weak Lefschetz property and that $d_1, \ldots, d_{n+1}$ are as in Theorem (\ref{main_theorem}). 
\begin{itemize}
\item If $n \ge 9$ then we must have $q \le 2$.

\item If $n \in \{7, 8\}$ then we must have $q \le 3$.

\item If $n\in \{5, 6\}$ then we must have $q \le 4$ (and therefore $p \le 3$).
\end{itemize}
\end{lemma}
\begin{proof}

{\bf Case 1:} $\sum _{i=1}^{n+1} k_i -n +1$ is odd. 
Let $I=1$ in equation (\ref{genI}). We have
\begin{equation}\label{I=1}
(n-2)q + r_{n+1} \le r_1 + \ldots + r_n + n-1
\end{equation}
If $n$ is odd, we let $I=n$ in equation (\ref{genI}), and we have
\begin{equation}\label{I=n}
-q + r_2 + \ldots + r_{n+1} \le r_1 + n-1
\end{equation}
Combining equations (\ref{I=1}) and (\ref{I=n}) we see that
$$
(n-2) q \le r_1 + (r_2 + \ldots + r_n) -r_{n+1} + n-1 \le q + 2r_1 -2r_{n+1} + 2(n-1) \le q + 2(n-1)
$$
and therefore $q \le 2(n-1)/(n-3)$. For $n\ge 9$, this implies that $q \le 2$. For $n=7$ we have $q \le 3$ and for $n=5$ we have $q \le 4$. 

If $n$ is even we let $I=n+1$ in equation (\ref{genI}). We have
\begin{equation}\label{I=n+1}
-2q + r_1 + \ldots + r_{n+1} \le n-1
\end{equation}
Combining equations (\ref{I=1}) and (\ref{I=n+1}) we obtain 
$$
(n-2) q \le r_1 + \ldots + r_n - r_{n+1} + (n-1) \le 2q -2r_{n+1} + 2(n-1),
$$
or $(n-4)q \le 2(n-1-r_{n+1})$. Since we are assuming $r_1 \le \ldots \le r_{n+1}$, we must have either $r_{n+1} \ge 1$ or $r_1 =\ldots = r_{n+1} =0$. In the latter case, equation (\ref{I=1}) becomes $(n-2)q \le n-1$ which is only possible if $q=1$. Thus we may assume that $r_{n+1} \ge 1$, and we have $q \le 2(n-2)/(n-4)$. When $n\ge 10$ this implies $q\le 2$. When $n=8$ it implies $q \le 3$, and when $n=6$ it implies $q\le 4$.

{\bf Case 2:} $\sum_{i=1}^{n+1} k_i -n +1$ is even. 
Plug in $I=0$ in equation (\ref{genIeven}). We get
\begin{equation}\label{I=0}
(n-1) q \le r_1 + \ldots + r_{n+1} + n-1
\end{equation}
If $n$ is even plug in $I=n$ in equation (\ref{genIeven}). We get
\begin{equation}\label{I=neven}
-q+ r_2 + \ldots + r_{n+1} \le r_1 + n-1.
\end{equation}
Since $r_1\le \ldots \le r_{n+1}$, this implies that $-q+nr_1 \le r_1 + n-1$, or $r_1 \le 1 + q/(n-1)$.
Combining equations (\ref{I=0}) and (\ref{I=neven}), we get
$$
(n-1)q \le r_1 + (r_2 + \ldots + r_{n+1}) + n-1 \le q+ 2r_1 +2(n-1) \le q + \frac{2q}{n-1} +2n
$$
and therefore 
$$
q \le \frac{2n}{n-2-\frac{2}{n-1}}.
$$
When $n\ge 8$ this implies $q \le 2$, when $n=6$ it implies $ q\le 3$, and when $n=4$ it implies $q \le 5$.

If $n$ is odd, plug in $I=n+1$ in equation (\ref{genIeven}). We get
\begin{equation}\label{I=n+1even}
-2q + r_1 + \ldots + r_{n+1} \le n-1
\end{equation}
Combining equations (\ref{I=0}) and (\ref{I=n+1even}), we get
$(n-1) q \le 2q + 2(n-1)$, which implies $q \le 2(n-1)/(n-3)$. If $n \ge 9$ this implies $ q\le 2$. If $n=7$ then $q \le 3$ and if $n=5$ then $q \le 4$.
\end{proof}

\begin{lemma}\label{p=2}
Assume that $p=2$ and $d_1, \ldots, d_{n+1}$ are as in Theorem (\ref{main_theorem}). If $A$ has the weak Lefschetz property, then we must have $n \le 3$, or $n=4$ and $q=2$.
\end{lemma}
\begin{proof}
The assumption that $d_1, \ldots, d_{n+1}$ are as in Theorem (\ref{main_theorem}) implies that $k_1=\ldots = k_{n+1}=1$. 
Equation (\ref{pq}) implies that
\begin{equation}\label{2q}
2q \ge \lceil \frac{\sum_{i=1}^{n+1} d_i - n +1 }{2} \rceil = \lceil \frac{(n+1)q + \sum_{i=1}^{n+1} r_i - n +1}{2} \rceil
\end{equation}
Note that we cannot have $r_1 = \ldots r_{n+1}=0$, and therefore equation (\ref{2q}) implies $(n+1)q - n +2 \le 4q$, or $(n-3)q\le n-2$. This can only hold if $n\le 3$, or $n=4$ and $q=2$.
\end{proof}

\begin{lemma}\label{p=3}
Assume that $d_1, \ldots, d_{n+1}$ are as in Theorem (\ref{main_theorem}) with $p=q=3$, and $A$ has WLP. Then we must have $n \le 4$.
\end{lemma}
\begin{proof}
Assume by contradiction that $n \ge 5$.
We discuss two cases according to the parity of $\sum_{i=1}^{n+1} k_i -n +1$. 

{\bf Case 1:} Assume that $\sum_{i=1}^{n+1} k_i -n +1$ is odd. From equation (\ref{I=1}) we have
\begin{equation}\label{pis3}
3(n-2) +r_{n+1} \le r_1 + \ldots + r_ n-1 \le nr_{n+1} + n-1
\end{equation}
which implies $3(n-2) \le (n-1)(r_{n+1}+1)$ or $r_{n+1} \ge 3(n-2)/(n-1) -1$. Since $r_{n+1}$ is an integer, $n \ge 5$ implies that $r_{n+1} \ge 2$.
Now equation (\ref{pis3}) implies $r_1 + \ldots  + r_n \ge 2n-1$.

Since $k_i \ge 1$ for all $i \in \{1, \ldots, n+1\}$, equation (\ref{pq}) implies that
$$
9 \ge \lceil \frac{3\sum_{i=1}^{n+1} k_i + \sum_{i=1}^{n+1} r_i -n +1}{2} \rceil \ge \lceil \frac{3(n+1) + 2n -1 + n-1}{2} \rceil =\lceil \frac{4n+3}{2}\rceil,
$$
and this implies $4n \le 15$, therefore $n \le 4$.

{\bf Case 2:} Assume that $\sum_{i=1}^{n+1} k_i -n +1$ is even. From equation (\ref{I=1}) we have
$r_1 + \ldots + r_{n+1} + n-1 \le 3(n-1),$
therefore $r_1 + \ldots + r_{n+1} \ge 2(n-1)$. As in the discussion of the previous case, we have
$$
9 \ge \lceil \frac{3\sum_{i=1}^{n+1} k_i + \sum_{i=1}^{n+1} r_i - n +1 }{2} \rceil \ge \lceil \frac{3(n+1) +2(n-1) - n+1}{2} \rceil = \lceil \frac{4n+2}{2}\rceil,
$$
which implies that $n \le 4$ as desired.

\end{proof}

Putting the results of Lemma(\ref{largen}), Lemma(\ref{p=2}) and Lemma(\ref{p=3}) together, we obtain the following:
\begin{corollary}\label{corWLP}
Assume that $n \ge 5$ and $d_i = k_i q + r_i$ are as above with $k_i \ge 1$. If $A$ has WLP, then we must have $q=1$.
\end{corollary}

Now we will consider the cases $n=4$ and $n=3$.

\begin{prop}\label{propWLP}
Assume that $n=4$, $d_1,  \ldots, d_5$ are as in Theorem (\ref{main_theorem}) with $q >1$, and $A$ has WLP. Then we must have $q=p=3$ and $(d_1, \ldots, d_5)=(4, 4, 4, 4, 5)$ or any permutation of this.
\end{prop}
\begin{proof}
We consider two cases according to the parity of $\sum_{i=1}^5 k_i-3$.

{\bf Case 1:} $\sum_{i=1}^5 k_i-3$ is odd. In order for the assumption in Theorem (\ref{main_theorem}) to hold, we must have $\sum_{i=1}^5 k_i \ge 8$.

By plugging $I=1, 3, 5$ in equation (\ref{genI}) we have:
\begin{equation}\label{eq1}
2q+r_5 \le r_1+r_2+r_3+r_4 +3
\end{equation}
\begin{equation}\label{eq2}
r_3+r_4+r_5\le r_1 + r_2 + 3 
\end{equation}
\begin{equation}\label{eq3}
-2q+r_1 + r_2+ r_3 +r_4 +r_5 \le 3
\end{equation}
Since $r_1\le r_2\le \ldots \le r_5$, equation (\ref{eq2}) implies $r_5 \le 3$, and therefore $r_1, \ldots, r_5 \le 3$. Using equations (\ref{eq1}) and (\ref{eq3}), we have

\begin{equation}\label{comp}
 r_1 + r_2 + r_3 + r_4 -(3-r_5) \le 2q \le r_1 + r_2 + r_3 + r_4 +(3-r_5)
\end{equation}
The argument will be based on considering the possible values of $r_5$.
Assume that $r_5=3$. Then equation (\ref{eq2}) implies that $r_1 =r_2=r_3=r_4:=r \le 3$ and equation (\ref{comp}) implies that $q=2r$. Therefore $q$ is even, so $p=2$.
Furthermore, since $k_i \ge 1$ for all $i \in \{1, \ldots, 5\}$, we have from 
 equation (\ref{pq}):
$$
2q\ge \lceil \frac{q\sum_{i=1}^{5} k_i + \sum_{i=1}^{5} r_i -3 }{2} \rceil \ge \lceil \frac{5q+4r}{2}\rceil = \lceil \frac{7q}{2}\rceil.
$$
This is not possible.

Assume that $r_5=2$. Equation (\ref{eq1}) implies that $2q \le r_1 + r_2 + r_3 + r_4 +1 \le 9$, thus $q \le 4$ and $p \le 3$.
 Since $k_i \ge 1$ for all $i \in \{1, \ldots, 5\}$ and we are assuming that $\sum_{i=1}^5 k_i$ is even, equation (\ref{pq}) implies
$$
pq \ge \lceil \frac{6q + \sum_{i=1}^5 r_i -3}{2}\rceil \ge \lceil \frac{6q + 2q-2}{2}\rceil =4q -1
$$
(for the last inequality, we used the fact that $\sum_{i=1}^5 r_i \ge 2q+1$, which follows from equation (\ref{eq1})).
Since $p \le 3$, this implies that $q=1$.

Now assume $r_5=1$.Then equation (\ref{eq1}) implies $2q\le 6$, so $q\le 3$. From equation (\ref{eq1}) we have $r_1 + \ldots + r_5 \ge 2q -1$. Moreover, we have $\sum_{i=1}^5 k_i \ge 6$, and equation (\ref{pq}) implies that
$$
pq \ge \lceil \frac{6q + 2q-4}{2}\rceil =4q-2
$$
Since $p\le 3$, this implies $q=1$.

If $r_5=0$ then equation (\ref{eq1}) implies $2q\le 3$, which implies $q=1$.

{\bf Case 2:} $\sum_{i=1}^5 k_i -3$ is even. By plugging $I=0, 2, 4$ in equation (\ref{genIeven}) we obtain
\begin{equation}\label{eqq1}
3q \le r_1 + \ldots + r_5 + 3,
\end{equation}
\begin{equation}\label{eqq2}
q+r_4 + r_5 \le r_1+r_2 + r_3 + 3 , \ \mathrm{and}
\end{equation}
\begin{equation}\label{eqq3}
-q + r_2 + r_3 + r_4 + r_5 \le r_1 + 3
\end{equation}
Equation (\ref{eqq2}) implies that $q \le r_1 + 3$. Combining this with equation (\ref{eqq3}) we get
$r_1 + 3 \ge q \ge r_2 + r_3 + r_4 + r_5 -r_1 - 3\ge 2(r_2 + r_3) -r_1 -3$, and thus $r_1 + 3 \ge r_2 + r_3$, which implies $ r_3 \le 3$, and thus
$q\le 6$, which implies $p\le 5$.
Using the fact that $\sum_{i=1}^5 k_i \ge 5$ and $\sum_{i=1}^5 r_i \ge 3q-3$ (from equation (\ref{eqq1})), we see that equation (\ref{pq}) implies 
$$
pq \ge \lceil \frac{5q + 3q-6}{2}\rceil=4q-3
$$
This is possible when $p=q=3$ or $p=q=5$.

Assume $p=q=3$, and thus $r_1, \ldots, r_5 \le q-1=2$. Assume $r_1=2$. Then $r_2 = \ldots = r_5 =2$ and we see that the inequality (\ref{pq}) fails since $\sum_{i=1}^5 k_i \ge 5$.
Assume $r_1 =1$. From equation (\ref{eqq1}) we have $r_2 + r_3 + r_4 +r_5 \ge 5$, from equation (\ref{eqq2}) we have $r_4 + r_5 \le r_2 + r_3 + 1$, and from equation (\ref{eqq3}) we have $r_2 + r_3 + r_4 + r_5 \le 7$. The only values that satisfy all these conditions are $(r_1, r_2, r_3, r_4, r_5)=(1, 1, 1, 1, 2)$ or a permuation of the above (in which case all the required equations do hold).
Moreover, we must have $k_1 = \ldots = k_5=1$, because otherwise $\sum_{i=1}^5 k_i \ge 7$, and equation (\ref{pq}) fails. Finally, one can use Theorem (\ref{main_theorem}) to check that $\mathcal{E}_3(4, 4, 4, 4, 5)=9=\lceil \frac{4+4+4+4+5-3}{2}\rceil $, and therefore $A$ has WLP for $(d_1, \ldots, d_5)=(4, 4, 4, 4, 5)$ or any permuation of this.
 Assume $r_1 =0$. From equations (\ref{eqq1}) and (\ref{eqq3}), we must have $r_2 + r_3+ r_4 + r_5=6$, and from equation (\ref{eqq2}) we have $r_4 + r_5 \le r_2 + r_3$, which implies that $r_2=r_3=r_4=r_5$. These two conditions cannot hold simultaneusly, since $r_2, \ldots, r_5$ are integers.

Assume $p=q=5$. Then equations (\ref{eqq1}), (\ref{eqq2}) and (\ref{eqq3}) become
$$
r_1 + \ldots + r_5 \ge 12,  \ r_4 + r_5 +2\le r_1 + r_2 + r_3, \  r_2 + r_3 + r_4 + r_5 \le r_1 + 8.
$$ 
Combining the first and the last inequalities, we get $12-r_1 \le r_1 + 8$, or $r_1 \ge 2$. Moreover, we have $4r_1 \le r_2 + r_3 + r_4 + r_5 \le r_1 +8$, which implies $3r_1 \le 8$, and therefore $r_1 =2$ is the only possiblity. Then the second inequality above implies that $r_2 = r_3 =r_4 = r_5:=r$, and the other two inequalities combined imply $4r=10$, which is not possible.

 \end{proof}
\begin{prop}\label{prop2WLP}
Assume that $n=3$ and $d_1, d_2, d_3, d_4$ are as in Theorem (\ref{main_theorem}) with $q>1$ and $r_1 \le r_2 \le r_3 \le r_4$. Then $A$ has WLP if and only if one of the following holds:
\begin{itemize}
\item $\sum_{i=1}^4 k_i$ is odd, $r_1 + r_2 + r_3 - r_4 + 2 \ge q \ge r_2 + r_3 + r_4 - r_1 -2$, and $\displaystyle pq \ge \lceil \frac{\sum_{i=1}^4 d_i -2}{2}\rceil $.

\item $\sum_{i=1}^4 k_i$ is even, $2q-2 \le r_1 + r_2 + r_3 + r_4 \le 2q+2$, $r_3 + r_4 \le r_1 + r_2 + 2$, and $\displaystyle pq \ge \lceil \frac{\sum_{i=1}^4 d_i -2}{2}\rceil $.

\end{itemize}
\end{prop}
\begin{proof}

{\bf Case 1:} $\sum_{i=1}^4 k_i -2$ is odd. Assume that $A$ has WLP.
Plugging $I=1, 3$ in equation (\ref{genI}), we get
\begin{equation}\label{eqqq1}
q + r_4 \le r_1 + r_2 + r_3 + 2 \ \mathrm{and}
\end{equation}
\begin{equation}\label{eqqq2}
-q + r_2 + r_3 + r_4 \le r_1 + 2
\end{equation}
The last condition in the statement is equation (\ref{pq}).

Thus shows that the conditions in the statement are necessary for $A$ to have WLP. Now we need to see that these conditions are also sufficient. From Theorem (\ref{main_theorem}), we have
$$\mathcal{E}_p(d_1, d_2, d_3, d_4)=\mathrm{min}\{ \frac{\sum_{i=1}^4 k_i -1}{2} q + r_1 + r_2 + r_3 , \frac{\sum_{i=1}^4 k_i + 1}{2} q + r_1, pq\}.
$$
We need to check that $\mathcal{E}_p(d_1, d_2, d_3, d_4)\ge \lceil \frac{\sum_{i=1}^4 d_i -2}{2} \rceil$. In other words we need to check that each of the following inequalities holds:
\begin{equation}\label{in1}
\frac{\sum_{i=1}^4 k_i -1}{2} q + r_1 + r_2 + r_3 \ge  \lceil \frac{\sum_{i=1}^4 d_i -2}{2} \rceil
\end{equation}
\begin{equation}\label{in2}
 \frac{\sum_{i=1}^4 k_i + 1}{2} q + r_1 \ge  \lceil \frac{\sum_{i=1}^4 d_i -2}{2} \rceil
\end{equation}
\begin{equation}\label{in3}
pq\ge \lceil \frac{\sum_{i=1}^4 d_i -2}{2} \rceil
\end{equation}
The inequality (\ref{in1}) is equivalent to $q \le r_1 + r_2 + r_3 - r_4 +2$.
The inequality (\ref{in2}) is equilvalent to $q \ge r_2 + r_3 + r_4 -r_1 -2$.
The inequality (\ref{in3}) is part of the assumption.

{\bf Case 2:} $\sum_{i=1}^4 k_i -2$ is even. Assume that $A$ has WLP. We get the first two conditions in the statement by plugging $I=0$, $I=2$, $I=4$ in equation (\ref{genIeven}), and the last condition is equation (\ref{pq}). Now we wish to see that the conditions in the statement are also sufficient for $A$ to have WLP. 
From Theorem (\ref{main_theorem}), we have
$$\mathcal{E}_p(d_1, d_2, d_3, d_4)=\mathrm{min}\{\frac{\sum_{i=1}^4 k_i -2}{2} q + r_1 + r_2 + r_3 + r_4, \frac{\sum_{i=1}^4 k_i}{2}q + r_1 + r_2, \frac{\sum_{i=1}^4 k_i + 2}{2}q, pq\}
$$
We need to check that $\mathcal{E}_p(d_1, d_2, d_3, d_4)\ge \lceil \frac{\sum_{i=1}^4 d_i -2}{2} \rceil$. In other words we need to check that each of the following inequalities holds:
\begin{equation}\label{inn1}
\frac{\sum_{i=1}^4 k_i -2}{2} q + r_1 + r_2 + r_3+r_4 \ge  \lceil \frac{\sum_{i=1}^4 d_i -2}{2} \rceil
\end{equation}
\begin{equation}\label{inn2}
 \frac{\sum_{i=1}^4 k_i}{2} q + r_1+r_2 \ge  \lceil \frac{\sum_{i=1}^4 d_i -2}{2} \rceil
\end{equation}
\begin{equation}\label{inn3}
 \frac{\sum_{i=1}^4 k_i+2}{2}q\ge  \lceil \frac{\sum_{i=1}^4 d_i -2}{2} \rceil
\end{equation}
\begin{equation}\label{inn4}
pq\ge \lceil \frac{\sum_{i=1}^4 d_i -2}{2} \rceil
\end{equation}
Equation (\ref{inn1}) is equivalent to $r_1 + r_2 + r_3 + r_4 \ge 2q-2$. Equation (\ref{inn2}) is equivalent to $r_3 + r_4 \le r_1 + r_2 + 2$. Equation (\ref{inn3}) is equivalent to $r_1 + r_2 + r_3 + r_4  \le 2q+2$, and equation (\ref{inn4}) is part of the assumption.
\end{proof}

\end{document}